\documentclass[12pt]{article}
\usepackage{amsthm,amsfonts,amssymb,epsfig,graphics,amsmath,amsbsy,color,enumerate,tikz}

\theoremstyle{plain}
\newtheorem{theorem}{Theorem}

\newtheorem{proposition}{Proposition}
\newtheorem{remark}{Remark}
\newtheorem{definition}{Definition}

\newcommand{\tr}{\operatorname{tr}}
\newcommand{\ind}{\operatorname{Ind}}

\numberwithin{equation}{section}
\numberwithin{lemma}{section}
\numberwithin{theorem}{section}
\numberwithin{remark}{section}
\numberwithin{claim}{section}
\numberwithin{corollary}{section}
\numberwithin{proposition}{section}
\numberwithin{definition}{section}
\numberwithin{condition}{section}
\numberwithin{figure}{section}

\DeclareFontFamily{OT1}{pzc}{}
\DeclareFontShape{OT1}{pzc}{m}{it}{<-> s * [1.100] pzcmi7t}{}
\DeclareMathAlphabet{\mathpzc}{OT1}{pzc}{m}{it}

\title{Renormalized Oscillation Theory for Regular Linear non-Hamiltonian Systems}

\author{Peter Howard}

\begin{document}

\maketitle

\begin{abstract} 
In recent work, Baird et al. have generalized the definition of the 
Maslov index to paths of Grassmannian subspaces that are not 
necessarily contained in the Lagrangian Grassmannian
[T. J. Baird, P. Cornwell, G. Cox, C. Jones, and R. Marangell,
{\it Generalized Maslov indices for non-Hamiltonian systems},
SIAM J. Math. Anal. {\bf 54} (2022) 1623-1668]. Such an extension opens
up the possibility of applications to non-Hamiltonian systems
of ODE, and Baird and his collaborators have taken advantage of 
this observation to establish oscillation-type results for obtaining
lower bounds on eigenvalue counts in 
this generalized setting. In the current analysis, the author 
shows that renormalized oscillation theory, appropriately 
defined in this generalized setting, can be applied in a natural 
way, and that it has the advantage, as in the traditional
setting of linear Hamiltonian systems, of ensuring monotonicity
of crossing points as the independent variable increases 
for a wide range of system/boundary-condition combinations.
This seems to mark the first effort to extend the renormalized 
oscillation approach to the non-Hamiltonian setting. 
\end{abstract}

\section{Introduction}\label{introduction}

For values of $\lambda$ in a real interval $I \subset \mathbb{R}$, 
we consider first-order ODE systems 
\begin{equation} \label{nonhammy}
    \frac{dy}{dx} = A (x; \lambda) y, 
    \quad x \in (0, 1), \quad y(x; \lambda) \in \mathbb{R}^n, 
    \quad n \in \{2, 3, \dots \},
\end{equation}
subject to boundary conditions 
\begin{equation} \label{bc}
y(0) \in \mathpzc{p}, 
\quad y(1) \in \mathpzc{q}, 
\end{equation}
where for some $m \in \{1, 2, ..., n-1 \}$ $\mathpzc{p}$ denotes
a subspace of $\mathbb{R}^n$ with dimension $m$ and 
$\mathpzc{q}$ denotes a subspace of $\mathbb{R}^n$ with 
dimension $n - m$. Throughout the analysis, we will assume 
that for some fixed values $\lambda_1, \lambda_2 \in I$,
$\lambda_1 < \lambda_2$, 
$A \in C([0,1] \times [\lambda_1, \lambda_2], \mathbb{R}^{n \times n})$,
and for convenient reference we will denote this 
assumption {\bf (A)}. In addition, for our main result we will 
assume the following, in which we denote the entries of 
$A(x; \lambda)$ by $\{a_{ij} (x; \lambda)\}_{i, j = 1}^n$ : 

\medskip
\noindent
{\bf (B)} For each $i \in \{1, 2, \dots, n\}$, the entry $a_{ii} (x; \lambda)$
is independent of $\lambda$, and for all 
$i, j \in \{1, 2, \dots, n\}$, $i \ne j$, and all $\lambda \in [\lambda_1, \lambda_2]$,
the difference $a_{ij} (x; \lambda) - a_{ij} (x; \lambda_2)$ is independent of $x$.

\medskip

Our analysis is primarily motivated by the prospect of applying the 
generalized Maslov index theory of \cite{BCCJM2022} to systems 
(\ref{nonhammy}) arising when an evolutionary PDE such as a 
viscous conservation law is linearized about a traveling 
wave solution. In particular, suppose $\bar{u} (x - st)$ denotes
a viscous profile for the system 
\begin{equation} \label{vcl-eqn}
    u_t + f(u)_x = B u_{xx}, \quad u(x, t) \in \mathbb{R}^l,
    \quad l \in \mathbb{N},
\end{equation}
where for this motivating example we take $B$ to be a constant 
viscosity matrix. In a moving coordinate frame, we can view 
$\bar{u} (x)$ as a stationary solution for the system
\begin{equation*}
    u_t - s u_x + f(u)_x = B u_{xx},
\end{equation*}
and if we linearize about $\bar{u} (x)$ with $u = \bar{u} + v$
(and drop off nonlinear terms), we arrive at the linear 
system 
\begin{equation*}
    v_t + ((Df(\bar{u}) - s I) v)_x = B v_{xx},
\end{equation*}
with associated eigenvalue problem 
\begin{equation} \label{vcl-evp}
    -B \phi '' + ((D f(\bar{u}) - s I) \phi)_x = \lambda \phi,
\end{equation}
where $Df (\bar{u} (x))$ denotes the usual Jacobian matrix
for $f$ evaluated at the wave. 
Under quite general conditions, the stability of $\bar{u} (x)$
is determined by the eigenvalues of (\ref{vcl-evp}) 
(see, e.g., \cite{ZH1998}), motivating our interest in 
eigenvalue problems of the general form 
\begin{equation} \label{general-form}
   -B \phi '' + W(x) \phi' + V(x) \phi = \lambda \phi.  
\end{equation}
In order to place this system in the setting of 
(\ref{nonhammy}), we write $y = {y_1 \choose y_2}$
with $y_1 = \phi$ and $y_2 = B \phi'$, giving 
(\ref{nonhammy}) with $n = 2l$ and
\begin{equation} \label{vclA}
    A(x; \lambda)
    = \begin{pmatrix}
    0 & B^{-1} \\
    V(x) - \lambda I & W(x) B^{-1}.
    \end{pmatrix}
\end{equation}
In this case, we see that Assumption {\bf (A)} holds
as long as $B$ is invertible and $W, V \in C([0,1],\mathbb{R}^{l \times l})$,
while Assumption {\bf (B)} is immediate. An additional family of motivating
examples is discussed in Section \ref{higher-order-section}.

In this general setting, we will say that $\lambda$ is an eigenvalue of (\ref{nonhammy})-(\ref{bc})
provided there exists a solution $y (\cdot; \lambda) \in C^1 ([0, 1], \mathbb{R}^n)$
of (\ref{nonhammy})-(\ref{bc}), and as usual we will refer to the dimension of the 
space of all such solutions as the geometric multiplicity of $\lambda$. 
Our main goal is to show that a notion of renormalized oscillation 
theory (described below) can be used to obtain a lower bound on the number of eigenvalues
$\mathcal{N}_{\#} ([\lambda_1, \lambda_2])$ (counted {\it without} multiplicity)
that (\ref{nonhammy})-(\ref{bc}) has on an interval $[\lambda_1, \lambda_2]$. 
Under our relatively weak assumptions on the dependence of $A (x; \lambda)$
on $\lambda$, it's possible that 
the eigenvalues of (\ref{nonhammy}), as we've defined them, won't comprise a discrete
set on the interval $[\lambda_1, \lambda_2]$. In this case, our convention will 
be to take $\mathcal{N}_{\#} ([\lambda_1, \lambda_2]) = + \infty$, in which case
our lower bounds on $\mathcal{N}_{\#} ([\lambda_1, \lambda_2])$ will be taken to 
hold trivially. For a more nuanced perspective, developed in the setting of linear
Hamiltonian systems, we refer the reader to \cite{Elyseeva2020} and references 
therein.

Our primary tool for this analysis will be a generalization of the Maslov index
introduced in \cite{BCCJM2022}, and for the purposes of this introduction
we will start with a brief, intuitive discussion of this object 
(see Section \ref{maslov-section} for additional details and reference 
\cite{BCCJM2022} for a full development). Precisely, we focus on the hyperplane
setting discussed in Section 3.2 of \cite{BCCJM2022}. 

To begin, for any $n \in \mathbb{N}$ 
we denote by $Gr_n (\mathbb{R}^{2n})$ the Grassmannian comprising
the $n$-dimensional subspaces of $\mathbb{R}^{2n}$, and we let 
$\mathpzc{g}$ denote an element of $Gr_n (\mathbb{R}^{2n})$. The
space $\mathpzc{g}$ can be
spanned by a choice of $n$ linearly independent vectors in 
$\mathbb{R}^{2n}$, and we will generally find it convenient to collect
these $n$ vectors as the columns of a $2n \times n$ matrix $\mathbf{G}$, 
which we will refer to as a {\it frame} for $\mathpzc{g}$.  
We specify a metric on $Gr_n (\mathbb{R}^{2n})$ in terms of appropriate orthogonal projections. 
Precisely, let $\mathcal{P}_i$ 
denote the orthogonal projection matrix onto $\mathpzc{g}_i \in Gr_n (\mathbb{R}^{2n})$
for $i = 1,2$. I.e., if $\mathbf{G}_i$ denotes a frame for $\mathpzc{g}_i$,
then $\mathcal{P}_i = \mathbf{G}_i (\mathbf{G}_i^* \mathbf{G}_i)^{-1} \mathbf{G}_i^*$.
We take our metric $d$ on $Gr_n (\mathbb{R}^{2n})$ to be defined 
by 
\begin{equation*}
d (\mathpzc{g}_1, \mathpzc{g}_2) := \|\mathcal{P}_1 - \mathcal{P}_2 \|,
\end{equation*} 
where $\| \cdot \|$ can denote any matrix norm. We will say 
that a path of Grassmannian subspaces 
$\mathpzc{g}: [a, b] \to \Lambda (n)$ is continuous provided it is 
continuous under the metric $d$. 

Given a continuous path of Grassmannian subspaces 
$\mathpzc{g}: [a, b] \to Gr_n (\mathbb{R}^{2n})$ 
and a fixed {\it target} space $\mathpzc{q} \in Gr_n (\mathbb{R}^{2n})$,
the generalized Maslov index of \cite{BCCJM2022} 
(under some additional conditions discussed below)
provides a means of counting intersections
between the the subspaces $\mathpzc{g} (t)$ and $\mathpzc{q}$
as $t$ increases from $a$ to $b$, counted
with direction, but not with multiplicity. (By multiplicity, we mean the 
dimension of the intersection; direction will be discussed in 
detail in Section \ref{maslov-section}). In order to understand how this
works, we first recall the notion of a kernel for a skew-symmetric
$n$-linear map $\omega$. 

\begin{definition}
For a skew-symmetric $n$-linear map
$\omega: \mathbb{R}^{2n} \times \dots \times \mathbb{R}^{2n} \to \mathbb{R}$
($\mathbb{R}^{2n}$ appearing $n$ times), we define the kernel, $\ker \omega$,
to be the subset of $\mathbb{R}^{2n}$,
\begin{equation*}
    \ker \omega := \{v \in \mathbb{R}^{2n}: \omega (v, v_1, \dots, v_{n-1}) = 0,
    \quad \forall \, v_1, v_2, \dots, v_{n-1} \in \mathbb{R}^{2n}\}.
\end{equation*}
\end{definition}

Given a target space $\mathpzc{q} \in Gr_n (\mathbb{R}^{2n})$, we first identify 
a skew-symmetric $n$-linear map $\omega_1$ so that $\mathpzc{q} = \ker \omega_1$.
For example, if we let $\{q_i\}_{i=1}^n$ denote a basis for $\mathpzc{q}$, then 
we can set 
\begin{equation*}
    \omega_1 (g_1, \dots, g_n)
    := \det (g_1 \,\, \dots \,\, g_n \,\, q_1 \dots q_n).
\end{equation*}
Next, we let $\omega_2$ denote any skew-symmetric $n$-linear map for which 
$\ker \omega_2 \ne \mathpzc{q}$, and we set 
\begin{equation*}
    \mathcal{H}_{\omega_i} 
    := \{\mathpzc{g} \in Gr_n (\mathbb{R}^{2n}): \mathpzc{g} \cap \ker \omega_i \ne \{0\}\},
    \quad i = 1, 2.
\end{equation*}
Then according to Definition 1.3 in \cite{BCCJM2022}, the set 
\begin{equation} \label{hyperplane-ma}
    \mathcal{M} := Gr_n (\mathbb{R}^{2n}) \backslash (\mathcal{H}_{\omega_1} \cap \mathcal{H}_{\omega_2})
\end{equation}
is a {\it hyperplane Maslov-Arnold space}.

\begin{definition} \label{invariance-definition}
We say that the flow $t \mapsto \mathpzc{g} (t)$
is {\it invariant} on $[a, b]$ with respect to $\omega_1$
and $\omega_2$ provided the values
\begin{equation*}
    \omega_1 (g_1 (t), \dots, g_n (t))
    \quad \text{and} \quad
    \omega_2 (g_1 (t), \dots, g_n (t))
\end{equation*}
do not simultaneously vanish at any $t \in [a, b]$ (i.e., 
$\mathpzc{g} (t) \in \mathcal{M}$ for all $t \in [a, b]$). For brevity, 
we say that the triple $(\mathpzc{g} (\cdot), \omega_1, \omega_2)$
is invariant on $[a, b]$. 
Likewise, we say that a map 
$\mathpzc{g}: [a,b] \times [c, d] \to Gr_n (\mathbb{R}^{2n})$
is invariant on $[a,b] \times [c, d]$ with respect to $\omega_1$
and $\omega_2$ provided the values 
\begin{equation} \label{quantities}
    \omega_1 (g_1 (s, t), \dots, g_n (s, t))
    \quad \text{and} \quad
    \omega_2 (g_1 (s, t), \dots, g_n (s, t))
\end{equation}
do not simultaneously vanish at any 
$(s, t) \in [a,b] \times [c, d]$ (i.e., 
$\mathpzc{g} (s, t) \in \mathcal{M}$ for all $(s,t) \in [a, b] \times [c, d]$). 
For brevity, we say that the triple $(\mathpzc{g} (\cdot, \cdot), \omega_1, \omega_2)$
is invariant on $[a, b] \times [c, d]$. Finally, we will say 
that a map $\mathpzc{g}: [a,b] \times [c, d] \to Gr_n (\mathbb{R}^{2n})$
is invariant on the boundary of $[a,b] \times [c, d]$ with respect to $\omega_1$
and $\omega_2$ provided the values in (\ref{quantities}) do not 
simultaneously vanish at any point $(s, t)$ on the boundary of 
$[a,b] \times [c, d]$. 
\end{definition}

\begin{remark} The terminology ``invariant" is taken 
from \cite{BCCJM2022}, where it arises naturally as the 
condition that a path in $P (\bigwedge^n (\mathbb{R}^{2n}))$
(i.e., the projective space of all one-dimensional subspaces
of the wedge space $\bigwedge^n (\mathbb{R}^{2n})$) 
associated to the flow $t \mapsto \mathpzc{g} (t)$ lies 
entirely in the {\it Maslov-Arnold space} introduced 
in \cite{BCCJM2022}. While this notion of the Maslov-Arnold
space is critical to the development of \cite{BCCJM2022},
we will only use it indirectly here, and so will omit a precise
definition. 
\end{remark}

In the event that the flow $t \mapsto \mathpzc{g} (t)$
is invariant on $[a, b]$ with respect to $\omega_1$
and $\omega_2$, the generalized Maslov index of \cite{BCCJM2022} can be computed 
as the winding number in projective space $\mathbb{R}P^1$ of the map 
\begin{equation} \label{extended-maslov}
    t \mapsto [\omega_1 (g_1 (t), \dots, g_n (t)) : \omega_2 (g_1 (t), \dots, g_n (t))]
\end{equation}
through $[0 : 1]$ (with appropriate conventions taken for counting arrivals and 
departures; see Section \ref{maslov-section} below). 
Following the convention of \cite{BCCJM2022}, we denote 
the generalized Maslov index as $\ind(\cdots)$, though our specific notation 
is adapted from \cite{HS2018, HS2021}, leading to
$\ind (\mathpzc{g} (\cdot), \mathpzc{q}; [a, b])$; i.e., 
$\ind (\mathpzc{g} (\cdot), \mathpzc{q}; [a, b])$ is a directed count 
of the number of times the subspace $\mathpzc{g} (t)$ has non-trivial 
intersection with $\mathpzc{q}$, counted without multiplicity, as $t$ 
increases from $a$ to $b$. 

For many applications, we would like to compute the generalized Maslov index
associated with a pair of evolving spaces $\mathpzc{g}, \mathpzc{h}: [a, b] \to Gr_n (\mathbb{R}^{2n})$,
or more generally (as in the current setting) a pair of evolving 
spaces $\mathpzc{g}: [a, b] \to Gr_m (\mathbb{R}^{n})$ and 
$\mathpzc{h}: [a, b] \to Gr_{n-m} (\mathbb{R}^{n})$, where
$m \in \{1, 2, \dots, n-1\}$. Following the approach of Section 3.5
in \cite{Furutani2004}, we can proceed by specifying an evolving 
subspace $\mathpzc{f}: [a, b] \to Gr_{n} (\mathbb{R}^{2n})$ 
with frame 
\begin{equation*}
    \mathbf{F} (t)
    := 
    \begin{pmatrix}
        \mathbf{G} (t) & \mathbf{0}_{n \times (n - m)} \\
        \mathbf{0}_{n \times m} & \mathbf{H} (t)
    \end{pmatrix},
\end{equation*}
and taking as the (fixed) target space the subspace 
$\tilde{\Delta} \in Gr_n (\mathbb{R}^{2n})$ with frame 
$\mathbf{\tilde{\Delta}} = {-I_n \choose I_n}$.
(Here, $\mathbf{G} (t)$ and $\mathbf{H} (t)$ are respectively 
frames for $\mathpzc{g} (t)$ and $\mathpzc{h} (t)$.)
We then specify the generalized Maslov index for the 
pair $\mathpzc{g}, \mathpzc{h}: [a, b] \to Gr_n (\mathbb{R}^{2n})$
to be 
\begin{equation} \label{extended-pairs}
    \ind (\mathpzc{g} (\cdot), \mathpzc{h} (\cdot); [a, b])
    := \ind (\mathpzc{f} (\cdot), \tilde{\Delta}; [a, b]),
\end{equation}
where the right-hand side is computed precisely as 
specified above (i.e., as in \cite{BCCJM2022}). 

\begin{remark} \label{frames-remark} Here, and throughout, we will
be as consistent as possible with the following notational 
conventions: we will 
express Grassmannian subspaces with script letters such as $\mathpzc{g}$,
and we will denote a choice of basis elements for $\mathpzc{g}$
by $\{g_i\}_{i=1}^m$. We will also collect these basis elements 
into an associated frame 
\begin{equation*}
    \mathbf{G} = (g_1, \,\, g_2, \,\, \dots, \,\, g_m).
\end{equation*}
\end{remark}

Returning to (\ref{nonhammy}), we begin by letting 
$\mathbf{G} (x; \lambda) \in \mathbb{R}^{n \times m}$ 
denote a matrix solution of the system 
\begin{equation} \label{g-frame}
    \mathbf{G}' = A(x; \lambda) \mathbf{G},
    \quad \mathbf{G} (0; \lambda) = \mathbf{P},
\end{equation}
where $\mathbf{P}$ denotes any frame for the 
subspace $\mathpzc{p}$ from (\ref{bc}), and likewise
we let $\mathbf{H} (x; \lambda) \in \mathbb{R}^{n \times (n-m)}$ 
denote a matrix solution of the system 
\begin{equation} \label{h-frame}
    \mathbf{H}' = A(x; \lambda) \mathbf{H},
    \quad \mathbf{H} (1; \lambda) = \mathbf{Q},
\end{equation}
where $\mathbf{Q}$ denotes any frame for the 
subspace $\mathpzc{q}$ from (\ref{bc}), and we 
emphasize that $\mathbf{H} (x; \lambda)$ is initialized
at $x=1$. 
Correspondingly, we let $\mathpzc{g} (x; \lambda)$
denote the $m$-dimensional subspace of $\mathbb{R}^n$
with frame $\mathbf{G} (x; \lambda)$, and we let 
$\mathpzc{h} (x; \lambda)$
denote the $(n-m)$-dimensional subspace of $\mathbb{R}^n$
with frame $\mathbf{H} (x; \lambda)$. 

Next, we fix any interval $[\lambda_1, \lambda_2] \subset I$,
$\lambda_1 < \lambda_2$, and for any $\lambda \in [\lambda_1, \lambda_2]$,
we set 
\begin{equation} \label{f-frame}
    \mathbf{F} (x; \lambda)
    := 
    \begin{pmatrix}
        \mathbf{G} (x; \lambda) & \mathbf{0}_{n \times (n - m)} \\
        \mathbf{0}_{n \times m} & \mathbf{H} (x; \lambda_2)
    \end{pmatrix} \in \mathbb{R}^{2n \times n},
\end{equation}
and correspondingly let $\mathpzc{f} (x; \lambda)$
denote the $n$-dimensional subspace of $\mathbb{R}^{2n}$
with frame $\mathbf{F} (x; \lambda)$. (We note that 
in the specification of $\mathbf{F} (x; \lambda)$, 
the frame $\mathbf{H}$ is evaluated at $(x, \lambda_2)$.)

In order to compute the generalized Maslov index specified in 
(\ref{extended-pairs}), we introduce the skew-symmetric 
$n$-linear map
\begin{equation} \label{omega1}
    \omega_1 (f_1, f_2, \dots, f_n)
    := \det (\mathbf{F} \,\,\, \tilde{\mathbf{\Delta}}),
\end{equation}
where $\{f_j\}_{j=1}^n \subset \mathbb{R}^{2n}$ 
comprise the columns of the $2n \times n$ matrix 
$\mathbf{F}$. With this specification, it's clear 
that $\ker \omega_1 = \tilde{\Delta}$. In order to use 
the development of \cite{BCCJM2022}, we additionally 
need to introduce any skew-symmetric $n$-linear map
$\omega_2$ for which $\ker \omega_2 \ne \tilde{\Delta}$. 
In principle, we have considerable freedom in the 
selection of $\omega_2$, but in practice we would 
like to choose $\omega_2$ in a specific way so that 
all crossing points for the generalized Maslov index
will have the same direction. Toward this end, we
specify $\omega_2$ in the following way. 

\medskip
{\it Specification of $\omega_2$}. 
Recalling that we denote by $\{a_{ij} (x; \lambda)\}_{i,j = 1}^n$ 
the components of the matrix $A(x; \lambda)$ from (\ref{nonhammy}),
we let $\tilde{A} (\lambda)$ denote the 
real-valued $n \times n$ matrix with entries 
\begin{equation*}
    \tilde{a}_{i j} (\lambda) :=  
    \begin{cases}
    0 & i = j \\
    a_{ij} (0; \lambda) & i \ne j
    \end{cases},
\end{equation*}
and set 
\begin{equation} \label{mathbb-A-tilde}
    \tilde{\mathbb{A}} (\lambda_1, \lambda_2)
    := 
    \begin{pmatrix}
    \tilde{A} (\lambda_1) & 0 \\
    0 & \tilde{A} (\lambda_2)   
    \end{pmatrix}.
\end{equation}
Then we define the skew-symmetric $n$-linear map
\begin{equation} \label{omega2-defined}
    \omega_2 (f_1, \dots, f_n)
    := \sum_{k=1}^n \omega_1 (f_1, \dots, \tilde{\mathbb{A}}(\lambda_1, \lambda_2) f_k, \dots, f_n).
\end{equation}
\medskip

Given $\omega_1$ as specified in (\ref{omega1}), and $\omega_2$ such that 
$\ker \omega_2 \ne \tilde{\Delta}$ (not necessarily as in (\ref{omega2-defined})), 
we will be particularly interested in 
computing the generalized Maslov index along the boundary of 
$[0, 1] \times [\lambda_1, \lambda_2]$ (see Figure \ref{box-figure}, 
below, in which we follow a long-standing convention of taking the axis
associated with the spectral parameter to be horizontal). 
Following the notation of \cite{BCCJM2022}, we will denote 
this quantity $\mathfrak{m}$, and precisely it follows from path 
additivity of the generalized Maslov index (as discussed in Section
\ref{maslov-section}) that 
\begin{equation*}
\begin{aligned}
    \mathfrak{m} &= \ind (\mathpzc{g} (0; \cdot), \mathpzc{h} (0; \lambda_2); [\lambda_1, \lambda_2])
    + \ind (\mathpzc{g} (\cdot; \lambda_2), \mathpzc{h} (\cdot; \lambda_2); [0, 1]) \\
    & \quad - \ind (\mathpzc{g} (1; \cdot), \mathpzc{h} (x; \lambda_2); [\lambda_1, \lambda_2])
    - \ind (\mathpzc{g} (\cdot; \lambda_1), \mathpzc{h} (\cdot; \lambda_2); [0, 1]).
\end{aligned}
\end{equation*}
In the event that the triple $(\mathpzc{f} (\cdot; \cdot), \omega_1, \omega_2)$ is invariant
on the entirety of $[0, 1] \times [\lambda_1, \lambda_2]$ it follows by a homotopy argument 
that $\mathfrak{m} = 0$, but this need not be the case in general. 

We are now in a position to state our main theorem. 

\begin{theorem} \label{main-theorem}
For (\ref{nonhammy})-(\ref{bc}), suppose Assumptions {\bf (A)} hold for some interval 
$[\lambda_1, \lambda_2] \subset I$, $\lambda_1 < \lambda_2$,
and for each $(x, \lambda) \in [0, 1] \times [\lambda_1, \lambda_2]$,
let $\mathpzc{g} (x; \lambda)$, $\mathpzc{h} (x; \lambda)$, and 
$\mathpzc{f} (x; \lambda)$ be linear spaces with frames 
respectively specified in (\ref{g-frame}), (\ref{h-frame}),
and (\ref{f-frame}). In addition, let $\omega_1$ denote 
the skew-symmetric $n$-linear map specified in (\ref{omega1}).
If $\omega_2$ is any skew-symmetric $n$-linear map 
for which the triple $(\mathpzc{f} (\cdot; \cdot), \omega_1, \omega_2)$
is invariant on the boundary of $[0, 1] \times [\lambda_1, \lambda_2]$,
then 
\begin{equation} \label{main-theorem-inequality}
    \mathcal{N}_{\#} ([\lambda_1, \lambda_2])
    \ge |\ind (\mathpzc{g} (\cdot; \lambda_1), \mathpzc{h} (\cdot; \lambda_2); [0, 1]) + \mathfrak{m}|.
\end{equation}
If we additionally assume {\bf (B)}, and let $\omega_2$ be the particular 
skew-symmetric $n$-linear map specified in (\ref{omega2-defined}), 
then 
\begin{equation*}
    \ind (\mathpzc{g} (\cdot; \lambda_1), \mathpzc{h} (\cdot; \lambda_2); [0, 1])
    = \# \{x \in (0, 1]: \mathpzc{g} (x; \lambda_1) \cap \mathpzc{h} (x; \lambda_2) \ne \{0\} \},
\end{equation*}
where the right-hand side of this final relation indicates a direct count of the 
(necessarily) discrete number of values $x \in (0, 1]$ at which the subspaces
$\mathpzc{g} (x; \lambda_1)$ and $\mathpzc{h} (x; \lambda_2)$ intersect
non-trivially. 
\end{theorem}

\begin{remark} \label{main-theorem-remark}
We note that in this statement we don't require that $\ker \omega_2$ 
be different from $\ker \omega_1$. This is simply because if 
$\ker \omega_2 = \ker \omega_1$, then the triple 
$(\mathpzc{f} (\cdot; \cdot), \omega_1, \omega_2)$ is invariant 
on the boundary of $[0, 1] \times [\lambda_1, \lambda_2]$ if and 
only if $\omega_1 (f_1 (x; \lambda), \cdots, f_{n} (x; \lambda))$
is non-zero for all $(x, \lambda) \in \partial ([0, 1] \times [\lambda_1, \lambda_2])$. 
But in this case, both sides of (\ref{main-theorem-inequality})
must be zero, and so the statement holds trivially. 

In order to understand why $x = 1$ is included in the final count in 
Theorem \ref{main-theorem} while $x = 0$ is not, we note that 
the final assertion of the theorem is established by showing that
the crossing points in the calculation of 
$ \ind (\mathpzc{g} (\cdot; \lambda_1), \mathpzc{h} (\cdot; \lambda_2); [0, 1])$
are monotonically positive. By convention, a positive crossing at the left
endpoint of an interval does not contribute to the count, while a positive crossing 
at the right endpoint of an interval does. This notion of direction is
discussed in Section \ref{maslov-section}. 
\end{remark}

In the remainder of this introduction, we briefly discuss the development 
of renormalized oscillation theory, and set out a plan for the paper. 
For the former, the notion of renormalized oscillation theory was introduced 
in \cite{GST1996} in the context of single Sturm-Liouville
equations, and subsequently was developed in 
\cite{Teschl1996, Teschl1998} for Jacobi operators and 
Dirac operators. More recently, Gesztesy and Zinchenko
have extended these early results to the setting of 
singular Hamiltonian systems in the limit-point case \cite{GZ2017},
and the author and Alim Sukhtayev have shown how the Maslov
index can be used to further extend such results to the 
full range of cases from limit-point to limit-circle 
\cite{HS2018, HS2021}. The primary motivation for the original 
development of \cite{GST1996} seems to have been the prospect 
of counting eigenvalues in gaps between bands of essential 
spectrum (such counts being problematic in the (non-renormalized) 
oscillation case). (See \cite{Simon2005} for an expository discussion.) 
The analyses described above
are all in the context of Hamiltonian systems for which 
the eigenvalues under investigation are discrete, 
possibly in a gap of essential spectrum. Renormalized 
oscillation theory has also been developed in some cases
for which nonlinear dependence on the spectral parameter
$\lambda$ leads to a generalized notion of eigenvalues
introduced in \cite{BKH2012} as {\it finite} eigenvalues. 
For the development in this setting (restricted to the 
Hamiltonian case), see \cite{Elyseeva2020}. 
Finally, we mention that the novel aspect of the current analysis 
is that it seems to be the first effort to extend 
renormalized oscillation results to the 
non-Hamiltonian setting.

{\it Plan of the paper}. In Section \ref{maslov-section}, 
we discuss the generalized Maslov index of \cite{BCCJM2022}, 
with an emphasis on properties that will be necessary for 
our analysis, and in Section \ref{rot-section} we 
discuss the application of renormalized oscillation theory
in the current setting. In Section \ref{proof-section} we prove
Theorem \ref{main-theorem}, and in Section \ref{invariance-section}
we develop a framework for checking the invariance assumption
of Theorem \ref{main-theorem} and computing the value 
$\mathfrak{m}$ in particular cases. In Section
\ref{applications-section}, we consider two families of 
examples, along with specific implementations for 
three particular equations.

\section{Properties of the Generalized Maslov Index}
\label{maslov-section}

In this section, we emphasize properties of the generalized 
Maslov index that will have a role in our analysis, leaving  
a full development of the theory to \cite{BCCJM2022}. 
In particular, a proper discussion of this object requires
some items from algebraic topology that are (1) already 
covered clearly and concisely in \cite{BCCJM2022}; 
and (2) not critical to the development of our results. 
Aside from an occasional clarifying comment for interested 
readers, these items are omitted from the current 
discussion. 

As in the introduction, we let $\mathpzc{g}: [a, b] \to Gr_n (\mathbb{R}^{2n})$ 
denote a continuous path of Grassmannian subspaces, and we let 
$\mathpzc{q} \in Gr_n (\mathbb{R}^{2n})$ denote a fixed {\it target}
subspace. We let $\omega_1$ denote a skew-symmetric $n$-linear map
such that $\ker \omega_1 = \mathpzc{q}$, and we let $\omega_2$ denote
a second skew-symmetric $n$-linear map so that the triple
$(\mathpzc{g} (\cdot), \omega_1, \omega_2)$ satisfies 
the invariance property described in Definition \ref{invariance-definition}
on the interval $[a, b]$ (i.e., $\mathpzc{g} (t) \in \mathcal{M}$
for all $t \in [a, b]$, where $\mathcal{M}$ is as in (\ref{hyperplane-ma})). 
(Here, we note that $\mathpzc{q}$ is not
needed in the triple notation, since $\mathpzc{q}$ is determined by $\omega_1$.) 
Recalling that our notational convention is to fix a choice of 
frames $\mathbf{G} (t)$ for $\mathpzc{g} (t)$ with columns 
$\{g_i (t)\}_{i=1}^n$, we set 
\begin{equation} \label{tilde-omega-again}
    \tilde{\omega}_i (t) 
    := \omega_i (g_1 (t), g_2 (t), \dots, g_n (t)),
    \quad i = 1, 2.
\end{equation}
I.e., $\omega_i$ will consistently denote a skew-symmetric $n$-linear map,
and $\tilde{\omega}_i$ will consistently denote the evaluation of 
$\omega_i$ along a particular path mapping $[a, b]$ to $Gr_n (\mathbb{R}^{2n})$.

The generalized Maslov index $\ind (\mathpzc{g} (\cdot), \mathpzc{q}; [a, b])$
is then computed as described in (\ref{extended-maslov}), with appropriate
conventions for counting arrivals and departures to and from the point 
in projective space $[0:1]$ (described below). In practice, we proceed
by tracking a point $p(t) \in S^1$, which can be precisely specified as 
\begin{equation} \label{p-specified}
p(t) = 
\begin{cases}
\Big( \frac{\tilde{\omega}_2 (t)}{\sqrt{\tilde{\omega}_1 (t)^2 + \tilde{\omega}_2 (t)^2}}, 
  \frac{\tilde{\omega}_1 (t)}{\sqrt{\tilde{\omega}_1 (t)^2 + \tilde{\omega}_2 (t)^2}} \Big) & \tilde{\omega}_2 (t) \le 0 \\
- \Big( \frac{\tilde{\omega}_2 (t)}{\sqrt{\tilde{\omega}_1 (t)^2 + \tilde{\omega}_2 (t)^2}}, 
   \frac{\tilde{\omega}_1 (t)}{\sqrt{\tilde{\omega}_1 (t)^2 + \tilde{\omega}_2 (t)^2}} \Big)  & \tilde{\omega}_2 (t) > 0.
\end{cases}
\end{equation}
In the usual way, we think of mapping $\mathbb{R}P^1$ to the left half
of the unit circle and then closing to $S^1$ by equating the points 
$(0,1)$ and $(0,-1)$. It's clear that $t_*$ is a crossing point of the flow
if and only if $p (t_*) = (-1, 0)$, so the generalized Maslov index is computed 
as a count of the number of times the point $p(t)$ crosses $(-1, 0)$. 
We take crossings in the clockwise direction to be negative and crossings
in the counterclockwise direction to be positive. Regarding
behavior at the endpoints, if $p(t)$
rotates away from $(-1,0)$ in the clockwise direction as $t$ increases
from $0$, then the generalized Maslov index decrements by 1, while if $p(t)$ 
rotates away from $(-1,0)$ in the counterclockwise direction as $t$ increases
from $0$, then the generalized Maslov index does not change. Likewise, 
if $p(t)$ rotates into $(-1,0)$ in the 
counterclockwise direction as $t$ increases
to $1$, then the generalized Maslov index increments by 1, while if 
$p(t)$ rotates into $(-1,0)$ in the clockwise direction as $t$ increases
to $1$, then the generalized Maslov index does not change. Finally, 
it's possible that $p(t)$ will arrive 
at $(-1,0)$ for $t = t_*$ and remain at $(-1,0)$ as $t$ traverses
an interval. In these cases, the generalized
Maslov index only increments/decrements upon arrival or 
departure, and the increments/decrements are determined 
as for the endpoints (departures determined as with $t=0$,
arrivals determined as with $t = 1$). 

\begin{remark} \label{bccjm-p}
In \cite{BCCJM2022}, the authors view $S^1$ as a circle in $\mathbb{C}$, 
and make the specification
\begin{equation*}
    p(t) = (\frac{\tilde{\omega}_1 (t) - i \tilde{\omega}_2 (t)}{|\tilde{\omega}_1 (t) - i \tilde{\omega}_2 (t)|})^2.
\end{equation*}
This choice leads to precisely the same dynamics as those described above, and 
in particular to the same values of the generalized Maslov index. 
\end{remark}

We emphasize, as in the introduction, that in contrast with the
Maslov index in the setting of Lagrangian flow, 
the generalized Maslov index does not keep track of the dimensions 
of the intersections. 

To set some notation, we let $\omega_1$ and $\omega_2$ be as 
above, and denote 
by $\mathcal{P}_{\omega_1, \omega_2} ([a, b])$ 
the collection of all continuous paths $\mathpzc{g}: [a, b] \to Gr_n (\mathbb{R}^{2n})$ 
that are invariant with respect to the skew-symmetric $n$-linear maps
$\omega_1$ and $\omega_2$. The generalized Maslov index of \cite{BCCJM2022} has the following
properties (see Proposition 3.8 in \cite{BCCJM2022}). 

\medskip
\noindent
{\bf (P1)} (Path Additivity) If $\mathpzc{g} \in \mathcal{P}_{\omega_1, \omega_2} ([a, b])$
and $\mathpzc{q} = \ker \omega_1$, then for any $\tilde{a}, \tilde{b}, \tilde{c} \in [a, b]$, 
with $\tilde{a} < \tilde{b} < \tilde{c}$, we have
\begin{equation*}
\ind (\mathpzc{g} (\cdot), \mathpzc{q}; [\tilde{a}, \tilde{c}]) 
= \ind (\mathpzc{g} (\cdot), \mathpzc{q};[\tilde{a}, \tilde{b}]) 
+ \ind (\mathpzc{g} (\cdot), \mathpzc{q}; [\tilde{b}, \tilde{c}]).
\end{equation*}

\medskip
\noindent
{\bf (P2)} (Homotopy Invariance) If 
$\mathpzc{g}, \mathpzc{h} \in \mathcal{P}_{\omega_1, \omega_2} ([a, b])$ 
are homotopic in $\mathcal{M}$ with $\mathpzc{g} (a) = \mathpzc{h} (a)$ and  
$\mathpzc{g} (b) = \mathpzc{h} (b)$ (i.e., if $\mathpzc{g}, \mathpzc{h}$
are homotopic with fixed endpoints) then 
\begin{equation*}
\ind (\mathpzc{g} (\cdot), \mathpzc{q}; [a, b]) = \ind (\mathpzc{h} (\cdot), \mathpzc{q}; [a, b]).
\end{equation*}

\subsection{Direction of Rotation}
\label{direction-section}

One of the advantages of the renormalized oscillation approach
in the linear Hamiltonian setting is that it often leads to monotoncity 
in the calculation of the Maslov index as the independent variable
varies \cite{HS2018, HS2021}. In order to show that the same 
advantage can be obtained in the non-Hamiltonian setting, we
employ the approach of Section 4 in \cite{BCCJM2022} to analyze
the direction of flow. For this, 
our starting point is the observation that for $p(t)$ near $(-1,0)$,
the location of $p(t)$ can be tracked via the angle 
\begin{equation} \label{p-angle}
    \theta(t) = \pi + \tan^{-1} \frac{\tilde{\omega}_1 (t)}{\tilde{\omega}_2 (t)},
\end{equation}
with $\pi$ arising from our convention of placing crossings at $(-1,0)$.
By the monotonicity of $\tan^{-1} x$, the direction of $\theta(t)$ near
a value $t=t_*$ for which $\theta (t_*) = \pi$ is determined by 
the derivative of the ratio $r(t) = \frac{\tilde{\omega}_1 (t)}{\tilde{\omega}_2 (t)}$,
for which $r(t_*) = 0$. Precisely, if 
$r'(t_*) = \frac{\tilde{\omega}'_1 (t_*)}{\tilde{\omega}_2 (t_*)} < 0$
then the rotation of $p(t)$ is clockwise at $t_*$, while if $r'(t_*) > 0$ then 
the rotation is counterclockwise.

\subsection{Invariance and the Computation of $\mathfrak{m}$}
\label{invariance-subsection}

Given a triple $(\mathpzc{g} (\cdot), \omega_1, \omega_2)$, we would like
to be able to check the invariance property of Definition \ref{invariance-definition}
on a given interval $[a, b]$. One strategy for this, employed in \cite{BCCJM2022},
is to show that the quantity $\tilde{\omega}_1 (t)^2 + \tilde{\omega}_2 (t)^2$ is non-zero 
at $t = a$ and to verify by computing its rate of change that it cannot
become 0 at any $t \in [a,b]$. More precisely, the authors of 
\cite{BCCJM2022} introduce scaled variables 
\begin{equation} \label{scaled-variables}
    \psi_1 (t) := \frac{\tilde{\omega}_1 (t)}{d(t)};
    \quad \textrm{and} \quad 
    \psi_1 (t) := \frac{\tilde{\omega}_2 (t)}{d(t)},
\end{equation}
where 
\begin{equation} \label{d-defined}
d (t) = |g_1 (t) \wedge g_2 (t) \wedge \dots \wedge g_n (t)|
= \sqrt{\det \mathbb{G} (t)},
\end{equation}
with $\mathbb{G} (t)$ denoting the Gram matrix; i.e., the matrix
with entries
$(\mathbb{G} (t))_{ij} = (g_i (t), g_j (t))$.  (Here, 
$(\cdot, \cdot)$ denotes the usual Euclidean inner product.)
If we then set 
\begin{equation} \label{rho-defined}
    \rho (t) = \frac{1}{2} (\psi_1 (t)^2 + \psi_2 (t)^2),
\end{equation}
we can proceed similarly as described above, checking that
$\rho (0) > 0$ and verifying that $\rho' (t)$ is sufficiently 
bounded below so that in fact $\rho (t)$ is bounded away from 0
for all $t \in [a, b]$. This calculation clearly depends
critically on the choices of $\tilde{\omega}_1 (t)$ and 
$\tilde{\omega}_2 (t)$.
Details in the setting of our analysis of (\ref{nonhammy}) 
are carried out in Section \ref{invariance-section}. 

\begin{remark} \label{coordinate-free}
An advantage of the variables $\psi_1 (t)$ and $\psi_2 (t)$
from (\ref{scaled-variables}) is that they are invariant 
(up to a possible change of sign) under coordinate 
transformations. Precisely, if $\omega$ denotes
any skew-symmetric $n$-linear map, then the evaluation 
of $\omega$ on the columns of $\mathbf{G}$ (i.e., on the 
basis elements $\{g_i\}_{i=1}^n$ for $\mathpzc{g}$) and the 
evaluation of $\omega$ on the columns of $\mathbf{G} M$ 
for some invertible $n \times n$ matrix $M$ (i.e., on a new basis for 
$\mathpzc{g}$) are related by 
\begin{equation*}
    \omega((\mathbf{G}M)_1, (\mathbf{G}M)_2, \dots, (\mathbf{G}M)_2)
    = (\det M) \omega(g_1, g_2, \dots, g_n).
\end{equation*}
Likewise, 
\begin{equation*}
    |(\mathbf{G}M)_1 \wedge (\mathbf{G}M)_2 \wedge \dots \wedge (\mathbf{G}M)_n|
    = |\det M| |g_1 \wedge g_2 \wedge \dots \wedge g_n|.
\end{equation*}
Combining these observations, we see that if we set 
\begin{equation*}
    \Psi (g_1, g_2, \dots, g_n)
    := \frac{\omega (g_1, g_2, \dots, g_n)}{|g_1 \wedge g_2 \wedge \dots \wedge g_n|},
\end{equation*}
then 
\begin{equation*}
    \Psi((\mathbf{G}M)_1, (\mathbf{G}M)_2, \dots, (\mathbf{G}M)_2)
    = \frac{\det M}{|\det M|} \Psi (g_1, g_2, \dots, g_n).
\end{equation*}
\end{remark}

More generally, suppose $\mathpzc{g}: [a, b] \times [c, d] \to Gr_n (\mathbb{R}^{2n})$
is a continuous map, and for some $\mathpzc{q} \in Gr_n (\mathbb{R}^{2n})$
let $\omega_1$ be as in (\ref{omega1}), with also $\omega_2$ denoting any 
skew-symmetric $n$-linear map with $\ker \omega_2 \ne \mathpzc{q}$. 
In the current generalized setting, it may be the case that the triple 
$(\mathpzc{g} (\cdot, \cdot), \omega_1, \omega_2)$ is invariant on 
the boundary of $[a, b] \times [c, d]$, but not on the entirety of 
its interior. In this case, the generalized Maslov index computed
along the boundary of $[a, b] \times [c, d]$ is well-defined, and as 
in the introduction we denote it $\mathfrak{m}$. 

In \cite{BCCJM2022}, the authors introduce a method that in some cases can 
be used to compute $\mathfrak{m}$ from local information in the interior 
of $[a, b] \times [c, d]$. For rigorous statements, the interested reader is 
referred to Lemmas 4.9 and 4.10 in \cite{BCCJM2022}, but the main ideas
are as follows. Suppose the triple 
$(\mathpzc{g} (\cdot, \cdot), \omega_1, \omega_2)$ loses invariance at 
a point $(s_*, t_*)$ in the interior of $[a, b] \times [c, d]$, so that
in particular we have both $\tilde{\omega}_1 (s_*, t_*) = 0$ and 
$\tilde{\omega}_2 (s_*, t_*) = 0$. In addition, suppose the point 
$(s_*, t_*)$ lies on a spectral curve that can be expressed near 
$(s_*, t_*)$ as a function $s (t)$: i.e., $s(t)$ satisfies 
$\tilde{\omega}_1 (s(t), t) = 0$ for $t$ sufficiently close to 
$t_*$, and also $s (t_*) = s_*$. Upon differentiating the 
relation $\tilde{\omega}_1 (s(t), t) = 0$ with respect to $t$
(in cases in which $s(t)$ is sufficiently smooth to allow it),
we find 
\begin{equation*}
    \frac{\partial \tilde{\omega}_1}{\partial s} (s_*, t_*) s' (t_*)
    + \frac{\partial \tilde{\omega}_1}{\partial t} (s_*, t_*)  = 0.
\end{equation*}
In certain cases, arising both in \cite{BCCJM2022} and the current 
analysis, we have additionally that 
$\frac{\partial \tilde{\omega}_1}{\partial t} (s_*, t_*) = 0$,
and in such cases points $(s_*, t_*)$ at which invariance is lost
can be characterized by the condition that either 
$\frac{\partial \tilde{\omega}_1}{\partial s} (s_*, t_*) = 0$
or $s' (t_*) = 0$ (or both). We will see an illustration of this 
dynamic in Section \ref{example3-section}. 

In order to understand the second observation from \cite{BCCJM2022}
regarding points $(s_*, t_*)$ at which invariance is lost, we observe
that in some cases, again arising both in \cite{BCCJM2022} and the 
current analysis, the flow associated with the generalized Maslov
index will be monotonic on horizontal lines (as in the case of 
\cite{BCCJM2022}) or vertical lines (as in the current setting). 
(This difference between \cite{BCCJM2022} and the current analysis 
is entirely artificial, depending only on different choices of 
orientation of the axes.) For specificity of this discussion, 
we will focus on the case in which the flow is monotonically 
positive on vertical axes as $t$ increases. In this setting, 
suppose $(s_*, t_*)$ is a point in the interior of 
$[a, b] \times [c, d]$ at which invariance fails. Then by 
a homotopy argument we can determine the contribution associated
with the point to the value $\mathfrak{m}$ by considering a sufficiently
small box enclosing $(s_*, t_*)$ and not enclosing any other points
at which invariance is lost (under the assumption that the points
at which invariance is lost form a discrete set). Moreover, we can 
think of selecting boxes sufficiently narrow in the $s$-direction
so that any spectral curves passing through $(s_*, t_*)$ 
necessarily enter and exit the small box through its vertical 
sides (see Figure \ref{invariance-figure}). 
In this way, the contribution associated
with $(s_*, t_*)$ to $\mathfrak{m}$ is entirely determined by 
the manner in which the spectral curves passing through 
$(s_*, t_*)$ cross the vertical shelves of this box. Precisely,
the analogue to Lemma 2.10 in \cite{BCCJM2022} in our setting 
can be loosely stated as follows: if we let $i_-$ denote the 
number of spectral curves that strictly increase as $t$ 
increases {\it to} $t_*$ (i.e., $s(t)$ strictly increases as $t$ 
increases to $t_*$), and we let $i_+$ denote the 
number of spectral curves that strictly increase as $t$ 
increases {\it from} $t_*$, and in addition we assume that all curves
are strictly monotonic as $t$ increases to/from $t_*$, 
then the contribution to $\mathfrak{m}$
associated with $(s_*, t_*)$ will be $2 (i_+ - i_-)$. 
 
\begin{figure}[ht]
\begin{center}
\begin{tikzpicture}

\draw[thick, ->] (-10, -3) -- (-8, -3); 
\draw[thick, ->] (-9.8, -3.2) -- (-9.8, -1.2); 
\node at (-8.1, -3.3) {$s$};
\node at (-10.1, -1.3) {$t$};

\draw[thick, ->] (-6, -2) -- (-5, -2);
\draw[thick] (-5, -2) -- (-4, -2);
\draw[thick, ->] (-4, -2) -- (-4, 0);
\draw[thick] (-4, 0) -- (-4, 2);
\draw[thick, ->] (-4, 2) -- (-5, 2);
\draw[thick] (-5, 2) -- (-6, 2);
\draw[thick, ->] (-6, 2) -- (-6, 0);
\draw[thick] (-6, 0) -- (-6, -2);

\draw[thick] (-7, -1) .. controls (-4.2, 0) .. (-7, 1);
\filldraw[black] (-4.9,0) circle (2pt);
\node at (-5,-1) {$(s_*, t_*)$};
\node at (-5, 2.6) {$i_- = 1, i_+ = 0$};
\node at (-5,-2.6) {$\ind = -2$};

\draw[thick, ->] (-2, -2) -- (-1, -2);
\draw[thick] (-1, -2) -- (0, -2);
\draw[thick, ->] (0, -2) -- (0, 0);
\draw[thick] (0, 0) -- (0, 2);
\draw[thick, ->] (0, 2) -- (-1, 2);
\draw[thick] (-1, 2) -- (-2, 2);
\draw[thick, ->] (-2, 2) -- (-2, 0);
\draw[thick] (-2, 0) -- (-2, -2);

\draw[thick] (1, -1) .. controls (-1.8, 0) .. (1, 1);
\filldraw[black] (-1.1,0) circle (2pt);
\node at (-1,-1) {$(s_*, t_*)$};
\node at (-1, 2.6) {$i_- = 0, i_+ = 1$};
\node at (-1,-2.6) {$\ind = 2$};

\draw[thick, ->] (2, -2) -- (3, -2);
\draw[thick] (3, -2) -- (4, -2);
\draw[thick, ->] (4, -2) -- (4, 0);
\draw[thick] (4, 0) -- (4, 2);
\draw[thick, ->] (4, 2) -- (3, 2);
\draw[thick] (3, 2) -- (2, 2);
\draw[thick, ->] (2, 2) -- (2, 0);
\draw[thick] (2, 0) -- (2, -2);

\draw[thick] (1, -1.5) -- (5, 1.5);
\filldraw[black] (3,0) circle (2pt);
\node at (3,-1) {$(s_*, t_*)$};
\node at (3, 2.6) {$i_- = 1, i_+ = 1$};
\node at (3,-2.6) {$\ind = 0$};

\end{tikzpicture}
\end{center}
\caption{Local contributions to $\mathfrak{m}$.} \label{invariance-figure}
\end{figure}
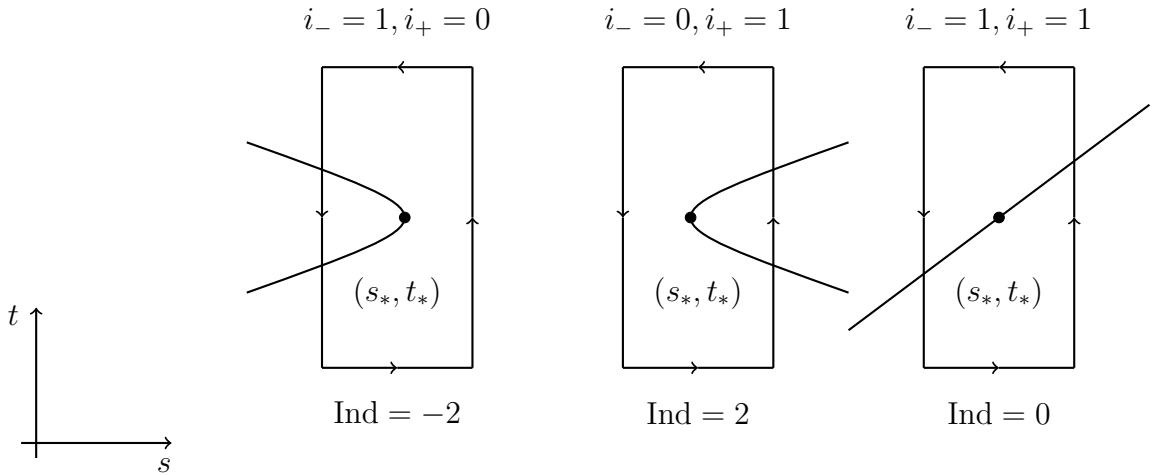 

\section{Oscillation Theory and Renormalized Oscillation Theory} 
\label{rot-section}

In \cite{BCCJM2022}, the authors use their generalized Maslov index to establish 
an oscillation result for systems (\ref{nonhammy}) arising from reaction-diffusion
systems 
\begin{equation} \label{reaction-diffusion}
    u_t = Bu_{xx} + F(u), 
    \quad u(x; t) \in \mathbb{R}^d,
\end{equation}
for which $F$ does not have a gradient structure (i.e., $F$ cannot be expressed
as the gradient of some map $\mathcal{F}: \mathbb{R}^n \to \mathbb{R}$). 
Similarly as with our discussion of (\ref{vcl-eqn}), we can naturally 
associate (\ref{reaction-diffusion}) with the eigenvalue problem 
\begin{equation} \label{reaction-diffusion-evp}
B \phi'' + DF(\bar{u} (x)) \phi = \lambda \phi,
\end{equation}
where $\bar{u} (x)$ denotes a stationary solution to (\ref{reaction-diffusion}). 
Equation (\ref{reaction-diffusion-evp}) can be expressed as (\ref{nonhammy}) with 
\begin{equation} \label{bccjm-matrix}
    A (x; \lambda) 
    = \begin{pmatrix}
        0 & B^{-1} \\
        \lambda I - DF (\bar{u} (x)) & 0 \\
    \end{pmatrix}.
\end{equation}

In order to compare the current approach with that of \cite{BCCJM2022}, we briefly 
summarize the main oscillation theorem from that reference (Theorem 4.1 in \cite{BCCJM2022}).
Considering (\ref{nonhammy}) on the interval $[0, L]$ for some $L > 0$,
with $A (x; \lambda)$ as specified in 
(\ref{bccjm-matrix}), the authors take boundary conditions at the right to be 
Dirichlet, and boundary conditions at the left to be either Dirichlet or 
Robin, where by Robin boundary conditions the authors mean that the space
$\mathpzc{p} \in Gr_n (\mathbb{R}^{2n})$ in (\ref{bc}) has a frame 
${I \choose \Phi}$, where $\Phi$ denotes any $n \times n$ matrix 
with real-valued entries. In order to express this result in the current 
framework and notation, we let $\mathbf{G} (x; \lambda)$ denote a 
$2n \times n$ matrix-valued solution of (\ref{nonhammy})-(\ref{bccjm-matrix}) 
such that $\mathbf{G} (0; \lambda) = \mathbf{P}$ (a frame for $\mathpzc{p}$),
and we let $\mathpzc{g} (x; \lambda)$ denote the evolving subspace with 
frame $\mathbf{G} (x; \lambda)$. With $\mathpzc{q}$ denoting the 
Dirichlet subspace, we specify $\omega_1$ so that $\ker \omega_1 = \mathpzc{q}$,
and set 
\begin{equation} \label{bccjm-omega2}
\omega_2 (g_1, \dots, g_n) 
:= \sum_{j=1}^n \omega_1 (g_1, \dots, A (x; \lambda) g_j, \dots, g_n).      
\end{equation}
Due to the particular form of $A(x; \lambda)$, $\omega_2$ does not 
explicitly depend on either $x$ or $\lambda$. For values $\delta > 0$
sufficiently small and $\lambda_{\infty} > 0$ sufficiently large, 
the authors of \cite{BCCJM2022} assume the triple 
$(\mathpzc{g} (\cdot; \cdot), \omega_1, \omega_2)$ is invariant on the 
boundary of the set $[\delta, L] \times [0, \lambda_{\infty}]$. Under these 
assumptions, the authors are able to conclude that 
\begin{equation*}
    \ind (\mathpzc{g} (L; \cdot), \mathpzc{q}; [0, \lambda_{\infty}])
    = \ind (\mathpzc{g} (\cdot; 0), \mathpzc{q}; [\delta, L]) - \mathfrak{m}. 
\end{equation*}
Here, the index on the left-hand side is a signed count of the number of 
eigenvalues that (\ref{nonhammy})-(\ref{bccjm-matrix}) (with the specified boundary conditions) 
has on the interval 
$[0, \lambda_{\infty}]$, and so cannot exceed a direct count of these eigenvalues; 
i.e., it must be the case that 
\begin{equation*}
    \mathcal{N}_{\#} ([0, \lambda_{\infty}])
    \ge |\ind (\mathpzc{g} (L; \cdot), \mathpzc{q}; [0, \lambda_{\infty}])|. 
\end{equation*}
In addition, the authors' choice of $\omega_2$, given here in (\ref{bccjm-omega2})
ensures that all crossing points for $\ind (\mathpzc{g} (\cdot; 0), \mathpzc{q}; [\delta, L])$
are positively directed, so that 
\begin{equation*}
    \ind (\mathpzc{g} (\cdot; 0), \mathpzc{q}; [\delta, L])
    = \# \{x \in (\delta, L]: \mathpzc{g} (x; 0) \cap \mathpzc{q} \ne \{0\} \}, 
\end{equation*}
where the count on the right-hand side is taken without multiplicity. 
Finally, the value $\lambda_{\infty}$ is taken large enough so that 
(\ref{nonhammy})-(\ref{bccjm-matrix}) (with the specified boundary conditions)
has no eigenvalues on the interval $[\lambda_{\infty}, \infty)$, and 
the value $\delta > 0$ is chosen sufficiently small so that 
\begin{equation*}
    \mathpzc{g} (x; 0) \cap \mathpzc{q} = \{0\},
    \quad \forall \, x \in (0, \delta). 
\end{equation*}
In this way, the conclusion of Theorem 4.1 of \cite{BCCJM2022} can be expressed
as 
\begin{equation*}
\mathcal{N}_{\#} ([0, \infty))
\ge | \# \{x \in (0, L]: \mathpzc{g} (x; 0) \cap \mathpzc{q} \ne \{0\} \} + \mathfrak{m}|.
\end{equation*}

This result is a natural generalization of standard oscillation results
for Sturm-Liouville systems, for which it's well known that in the 
case of a Dirichlet boundary condition on the right-hand side 
all crossing points as the independent variable increases 
will have the same sign. (See, e.g,. 
\cite{Arnold1985, BOPW2020, CB2015, CJLS2016, DDP2008, DJ2011, 
Edwards1964, Howard2021, HS2016, SB2012}). On the other hand, in 
both the Hamiltonian and non-Hamiltonian settings, if the target 
space is not Dirichlet then such monotonicity is not assured.   
As shown in \cite{HS2018, HS2021}, renormalized oscillation 
theory in the case of linear Hamiltonian systems leads naturally 
to a Maslov index for which all crossings as the independent variable
increases have the same sign, 
and so it's natural to ask if the same holds true in the current
non-Hamiltonian setting. The primary observation of 
Theorem \ref{main-theorem} in the current analysis is that it does.

\section{Proof of Theorem \ref{main-theorem}}
\label{proof-section}

We begin by fixing $\lambda_1, \lambda_2 \in I$, $\lambda_1 < \lambda_2$,
and letting $\mathbf{G} (x; \lambda)$ and $\mathbf{H} (x; \lambda_2)$
respectively denote the frames specified in (\ref{g-frame})
and (\ref{h-frame}), noting that $\mathbf{H}$ is evaluated at 
the fixed value $\lambda_2$. If $\mathbf{F} (x; \lambda)$ is 
specified as in (\ref{f-frame}) then $\mathbf{F} (x; \lambda)$
is a matrix solution to the ODE 
\begin{equation} \label{f-equation}
    \mathbf{F}' = \mathbb{A} (x; \lambda, \lambda_2) \mathbf{F},
    \quad  \mathbb{A} (x; \lambda, \lambda_2) 
    := \begin{pmatrix}
        A (x; \lambda) & 0_{n \times n} \\
        0_{n \times n} & A (x; \lambda_2)
    \end{pmatrix},
\end{equation}
though not to any particular initial value problem since $\mathbf{G} (x; \lambda)$
is initialized at $x = 0$ and $\mathbf{H} (x; \lambda_2)$ is initialized 
at $x=1$. Here, for each $(x, \lambda) \in [0,1] \times I$,
$\mathbf{F} (x; \lambda) \in \mathbb{R}^{2n \times n}$ is a 
frame for a subspace $\mathpzc{f} (x; \lambda) \in Gr_n (\mathbb{R}^{2n})$,
allowing us to compute the generalized Maslov index for the pair 
$\mathpzc{g} (x; \lambda)$ and $\mathpzc{h} (x; \lambda_2)$ by 
computing the generalized Maslov index for $\mathpzc{f} (x; \lambda)$
with target $\tilde{\mathbf{\Delta}} = {-I_n \choose I_n}$. 
(The frame $\mathbf{F} (x; \lambda)$ also depends on $\lambda_2$, 
but $\lambda_2$ remains fixed throughout the analysis, so this dependence
is suppressed.) 

As discussed in the introduction, we define the skew-symmetric $n$-linear map
\begin{equation} \label{omega1-form}
    \omega_1 (f_1, f_2, \dots, f_n)
    := \det (\mathbf{F} \,\,\, \tilde{\mathbf{\Delta}}),
    \quad \mathbf{F} = (f_1, f_2, \dots, f_n),
\end{equation}
and recalling our convention described in (\ref{tilde-omega-again}), the associated function 
\begin{equation} \label{tilde-omega1}
\begin{aligned}
    \tilde{\omega}_1 (x; \lambda) 
    &:= \det (\mathbf{F} (x; \lambda) \,\,\, \tilde{\mathbf{\Delta}}) \\
    &= \det \begin{pmatrix}
         \mathbf{G} (x; \lambda) & \mathbf{0}_{n \times (n-m)} & - I_{n} \\
         \mathbf{0}_{n \times m} & \mathbf{H} (x; \lambda_2) & I_{n}
    \end{pmatrix}
    = \det (\mathbf{G} (x; \lambda) \,\,\, \mathbf{H} (x; \lambda_2)).
\end{aligned}
\end{equation}
Next, we let $\omega_2$ denote any skew-symmetric $n$-linear map with 
$\ker \omega_2 \ne \tilde{\Delta}$ (though see Remark \ref{main-theorem-remark}), 
and we set 
\begin{equation} \label{tilde-omega2-defined}
    \tilde{\omega}_2 (x; \lambda)
    := \omega_2 (f_1 (x; \lambda), f_2 (x; \lambda), \dots, f_n (x; \lambda)). 
\end{equation}

\subsection{Proof of Theorem \ref{main-theorem}: First Claim}
\label{first-claim-section}

We will establish the first part of Theorem \ref{main-theorem}
by computing the generalized Maslov index for the pair $\mathpzc{g} (x; \lambda)$
and $\mathpzc{h} (x; \lambda_2)$ along the following sequence of contours,
often referred to as the {\it Maslov box}: (1) fix $x = 0$ and let $\lambda$
increase from $\lambda_1$ to $\lambda_2$ (the {\it bottom shelf}); (2) fix
$\lambda = \lambda_2$ and let $x$ increase from 0 to 1 (the {\it right shelf});
(3) fix $x = 1$ and let $\lambda$ decrease from $\lambda_2$ to $\lambda_1$
(the {\it top shelf}); and (4) fix $\lambda = \lambda_1$ and let $x$ decrease 
from 1 to 0 (the {\it left shelf}). See Figure \ref{box-figure}.

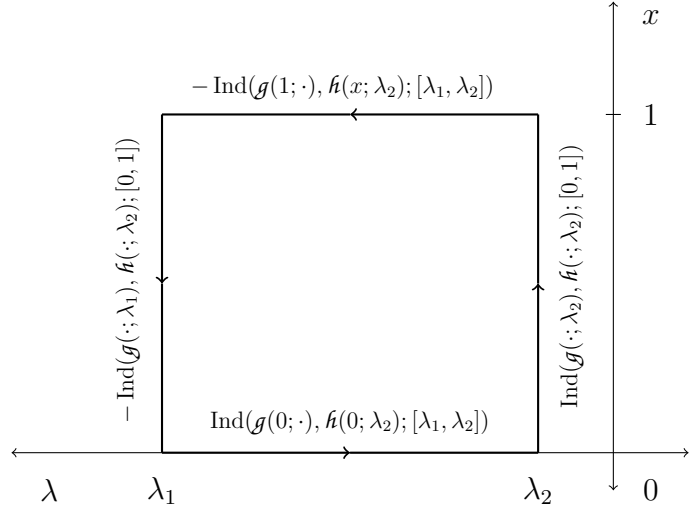
\begin{figure}[ht]
\begin{center}
\begin{tikzpicture}
\draw[<->] (-8,-1.5) -- (1,-1.5);	
\draw[<->] (0,-2) -- (0,4.5);	
\node at (.5,4.3) {$x$};
\node at (-7.5,-2) {$\lambda$};
\node at (-6,-2) {$\lambda_1$};
\node at (-1,-2) {$\lambda_2$};
\node at (.5,3) {$1$};
\draw[-] (-.1,3) -- (.1,3);
\node at (.5,-2) {$0$};
%
\draw[thick, ->] (-6,-1.5) -- (-3.5,-1.5);
\draw[thick] (-3.5,-1.5) -- (-1,-1.5);	
\draw[thick, ->] (-1,-1.5) -- (-1,.75);
\draw[thick] (-1,.75) -- (-1, 3);
\draw[thick,->] (-1,3) -- (-3.5, 3);
\draw[thick] (-3.5,3) -- (-6, 3);
\draw[thick,->] (-6,3) -- (-6, .75);
\draw[thick] (-6,.75) -- (-6, -1.5);
%
\node[scale = .75] at (-3.5, -1.1) {$\ind (\mathpzc{g} (0; \cdot), \mathpzc{h} (0; \lambda_2); [\lambda_1, \lambda_2])$};
\node[scale = .75, rotate=90] at (-.55, .85) {$\ind (\mathpzc{g} (\cdot; \lambda_2), \mathpzc{h} (\cdot; \lambda_2); [0, 1])$};
\node[scale = .75] at (-3.6, 3.35) {$- \ind (\mathpzc{g} (1; \cdot), \mathpzc{h} (x; \lambda_2); [\lambda_1, \lambda_2])$};
\node[scale = .75, rotate=90] at (-6.45, .8) {$- \ind (\mathpzc{g} (\cdot; \lambda_1), \mathpzc{h} (\cdot; \lambda_2); [0, 1])$};
\end{tikzpicture}
\end{center}
\caption{The Maslov Box.} \label{box-figure}
\end{figure}

{\it The right shelf}. We begin with the right shelf, observing that
for any $x \in [0, 1]$, $\omega_1 (x; \lambda_2)$ will be zero  
if and only if $\lambda_2$ is an eigenvalue of (\ref{nonhammy})
(because $\omega_1 (x; \lambda_2)$ will be zero if and only if 
$\mathpzc{g} (x; \lambda_2)$ and $\mathpzc{h} (x; \lambda_2)$
intersect non-trivially). If $\lambda_2$ is not an eigenvalue 
of (\ref{nonhammy}) then there can be no crossings along the right
shelf, and so trivially 
\begin{equation} \label{right-maslov}
    \ind (\mathpzc{g} (\cdot; \lambda_2), \mathpzc{h} (\cdot; \lambda_2); [0, 1]) = 0.
\end{equation}
On the other hand, if $\lambda_2$ is an eigenvalue of (\ref{nonhammy})
then every point on the right shelf is a crossing point. Since the Maslov
index only increases or decreases at arrivals and departures, this 
means that in fact (\ref{right-maslov}) holds in this case as well. We 
emphasize here that by our assumption of invariance along the boundary 
of the Maslov box, 
if $\lambda_2$ is an eigenvalue of (\ref{nonhammy}) so that 
$\omega_1 (x; \lambda_2) = 0$ for all $x \in [0, 1]$, 
then it must be the case that $\omega_2 (x; \lambda_2) \ne 0$
for all $x \in [0, 1]$. 

{\it The bottom shelf}. For the bottom shelf, $\mathbf{G} (0; \lambda) = \mathbf{P}$
for all $\lambda \in [\lambda_1, \lambda_2]$, so $\omega_1 (0; \lambda)$ and 
$\omega_2 (0; \lambda)$ do not vary with $\lambda$. In particular, 
$\omega_1 (0; \lambda)$ and $\omega_2 (0; \lambda)$ can both be evaluated 
at $\lambda = \lambda_2$ for all $\lambda \in [\lambda_1, \lambda_2]$, 
and in this way we see, as in our discussion of the right shelf, that 
if $\lambda_2$ is not an eigenvalue of (\ref{nonhammy}) then no point
on the bottom shelf is a crossing point, while if $\lambda_2$ is an 
eigenvalue of (\ref{nonhammy}) then every point on the bottom shelf is 
a crossing point. In either case, 
\begin{equation*}
    \ind (\mathpzc{g} (0; \cdot), \mathpzc{h} (0; \lambda_2); [\lambda_1, \lambda_2]) = 0.
\end{equation*}

{\it The top shelf}. Each crossing point along the top shelf corresponds with 
an eigenvalue of (\ref{nonhammy}), counted with direction, but not with 
multiplicity. Some crossing points may be positively directed while others are negatively 
directed, so there may be cancellation among these, leading to a value of 
the generalized Maslov index below (never above) the total number of eigenvalues. If 
we let $\mathcal{N}_{\#} ([\lambda_1, \lambda_2])$ denote the total number of 
eigenvalues that (\ref{nonhammy}) has on $[\lambda_1, \lambda_2]$, counted 
without multiplicity, then 
\begin{equation} \label{count-inequality}
    \mathcal{N}_{\#} ([\lambda_1, \lambda_2]) 
    \ge | \ind (\mathpzc{g} (1; \cdot), \mathpzc{h} (1; \lambda_2); [\lambda_1, \lambda_2]) |.
\end{equation}
As discussed in the introduction, we allow for the possibility that 
the left-hand side of (\ref{count-inequality}) is $+ \infty$, in which case 
we take (\ref{count-inequality}) to hold trivially, regardless of the value
of the right-hand side (which cannot be infinite by compactness of $[\lambda_1, \lambda_2]$,
and the observation that the point $p(1; \lambda) \in S^1$ that we track in computing
the generalized Maslov index must complete a full loop of $S^1$ before adding a
contribution to the generalized Maslov index with the same sign as the previous
contribution). 

{\it The left shelf}. For the first part of Theorem \ref{main-theorem}, the 
generalized Maslov 
index along the left shelf appears precisely in the original form 
\begin{equation*}
    \ind (\mathpzc{g} (\cdot; \lambda_1), \mathpzc{h} (\cdot; \lambda_2); [0, 1]),
\end{equation*}
so nothing additional is required until we turn to monotonicity in Section \ref{monotonicity-section} 
below.

Using path additivity, we can compute the generalized Maslov index along all four shelves of
the Maslov box to obtain the sum 
\begin{equation*}
    \mathfrak{m} = - \ind (\mathpzc{g} (1; \cdot), \mathpzc{h} (1; \lambda_2); [\lambda_1, \lambda_2])
    - \ind (\mathpzc{g} (\cdot; \lambda_1), \mathpzc{h} (\cdot; \lambda_2); [0, 1]),
\end{equation*}
Upon combining this relation with (\ref{count-inequality}), we immediately 
obtain the first claim of Theorem \ref{main-theorem},
\begin{equation*} 
    \mathcal{N}_{\#} ([\lambda_1, \lambda_2]) 
    \ge | \ind (\mathpzc{g} (\cdot; \lambda_1), \mathpzc{h} (\cdot; \lambda_2); [0, 1]) + \mathfrak{m}|.
\end{equation*}

\subsection{Proof of Theorem \ref{main-theorem}: Monotonicity}
\label{monotonicity-section}

Using the development of Section \ref{direction-section}, we see 
that the direction associated with a crossing point $x_*$ on the 
left shelf of the Maslov box is determined by the sign of 
$\frac{\tilde{\omega}_1' (x_*; \lambda_1)}{\tilde{\omega}_2 (x_*; \lambda_1)}$. 
Following the strategy of \cite{BCCJM2022}, we can ensure 
monotonicity of crossings by using our freedom with 
$\omega_2$ to choose it in such a way that $\tilde{\omega}_1' (x_*; \lambda_1)$
and $\tilde{\omega}_2 (x_*; \lambda_1)$ have the same sign for each 
crossing point $x_* \in [0, 1]$. As a starting point toward making such a selection,
we observe that for any $\lambda \in [\lambda_1, \lambda_2]$,
we have the relation
\begin{equation} \label{suppressed}
        \partial_x \omega_1 (f_1 (x; \lambda), \dots, f_n (x; \lambda)) 
        = \partial_x \det (g_1 (x; \lambda), \dots, g_m (x; \lambda), h_1 (x; \lambda_2), \dots, h_{n-m} (x; \lambda_2)).
\end{equation}

\begin{remark} \label{determinant-convention}
Here, and in subsequent calculations, notation such as 
\begin{equation*}
    \det (f_1 (x; \lambda), \dots, f_n (x; \lambda))
\end{equation*}
will indicate the determinant of the matrix comprising the vectors
$\{f_i (x; \lambda)\}_{i=1}^n$ as its columns in the indicated order. 
\end{remark}

Our approach to calculating derivatives of determinants of 
$n \times n$ matrices will primarily be to 
sum the $n$ terms obtained by putting a derivative on each of the $n$ 
different rows. For notational convenience we will write 
\begin{equation*}
\partial_x \omega_1 (f_1 (x; \lambda), \dots, f_n (x; \lambda)) 
= \sum_{i = 1}^n D_i (x; \lambda, \lambda_2),
\end{equation*}
where $D_i (x; \lambda, \lambda_2)$ is the determinant of the 
matrix obtained by replacing the $i^\textrm{th}$ row of 
$(\mathbf{G} (x; \lambda) \,\, \mathbf{H} (x; \lambda_2))$
with the associated row of derivatives (in $x$),
\begin{equation*}
      \begin{pmatrix}
        g_{i1}' (x; \lambda) & \cdots & g_{im}' (x; \lambda) & h_{i1}' (x; \lambda_2) & \cdots & h_{i(n-m)}' (x; \lambda_2) 
      \end{pmatrix}.
\end{equation*}

For calculations of this type, we will make use of the relations 
\begin{equation} \label{standard-relations}
\begin{aligned}
    g_{ij}' (x; \lambda) &= a_{ik} (x; \lambda) g_{kj} (x; \lambda) \\
    h_{ij}' (x; \lambda_2) &= a_{ik} (x; \lambda_2) g_{kj} (x; \lambda_2), 
\end{aligned}
\end{equation}
where we've streamlined notation slightly by assuming summation over the 
repeated index $k$. This allows us to replace the $i^\textrm{th}$ row of 
$(\mathbf{G} (x; \lambda) \,\, \mathbf{H} (x; \lambda_2))$ with 
\begin{equation} \label{Di-continued1}
        \begin{pmatrix}
        a_{ik} g_{k1} (x; \lambda) & \cdots & a_{ik} g_{km} (x; \lambda) 
        & a_{ik} h_{k1} (x; \lambda_2) & \cdots & a_{ik} h_{k(n-m)} (x; \lambda_2) 
        \end{pmatrix}.
\end{equation}
where we've made the additional reduction of notation $a_{ik} (x; \lambda) g_{k1} (x \lambda) = a_{ik} g_{k1} (x; \lambda)$,
and similarly for the other sums.
We can now use row operations to eliminate from column $j$, $j = 1, 2, \dots, m$, the 
sums $\sum_{k\ne i} a_{ik} (x; \lambda) g_{kj} (x; \lambda)$. For the remaining 
columns $j = m+1, \dots, n$ these row operations will lead to difference expressions, 
and combining these observations we can express 
$D_i (x; \lambda, \lambda_2)$ as the determinant of the 
matrix obtained by replacing the $i^\textrm{th}$ row of 
$(\mathbf{G} (x; \lambda) \,\, \mathbf{H} (x; \lambda_2))$
with 
\begin{equation} \label{Di-continued2}
        \begin{pmatrix}
        a_{ii} g_{i1} & \cdots & a_{ii} g_{im} 
        & a_{ii} h_{i1} + \mathcal{S}_i (\lambda, \lambda_2) h_1 & \cdots & a_{ii} h_{i(n-m)} + \mathcal{S}_i (\lambda, \lambda_2) h_{n-m} 
        \end{pmatrix},
\end{equation}
where dependence on $\lambda$ and $\lambda_2$ has been suppressed for typesetting
purposes (each term $a_{ii} g_{ij}$ is evaluated at $(x;\lambda)$ and each term
$a_{ii} h_{ij}$ is evaluated
at $(x; \lambda_2)$), and additionally we have introduced the notation
\begin{equation} \label{difference-sums}
    \mathcal{S}_i (\lambda,\lambda_2) h_j 
    := \sum_{k \ne i} (a_{ik} (x; \lambda_2) - a_{ik} (x; \lambda)) h_{k j},
    \quad i \in \{1, \cdots, n\}, 
    \quad j \in \{1, \dots, n-m\}.
\end{equation}
We note that in (\ref{Di-continued2}) triply-repeated indices {\it do not} 
indicate summation

Under Assumption {\bf{(B)}}, the entries $a_{ii} (x; \lambda)$ and 
$a_{ii} (x; \lambda_2)$ agree for each $i \in \{1, 2, \dots, n\}$,
and additionally the differences
$a_{ik} (x; \lambda) - a_{ik} (x; \lambda_2)$, 
$i, k \in \{1, 2, \dots, n\}$, $j \ne k$, in the specification of 
$\mathcal{S}_i (\lambda, \lambda_2) h_j$, $j = 1,2,\dots, n-m$ are independent of $x$.
These considerations allow us to write 
\begin{equation} \label{D1-defined}
    D_i (x; \lambda, \lambda_2) 
    := a_{ii} (x; \lambda) \tilde{\omega}_1 (x; \lambda)
    + \tilde{\omega}_2^i (x; \lambda, \lambda),
\end{equation}
where the slightly more general function $\tilde{\omega}_2^i (x; \lambda, \nu)$ is the determinant of 
the matrix obtained by replacing the $i^\textrm{th}$ row of 
$(\mathbf{G} (x; \lambda) \,\, \mathbf{H} (x; \lambda_2))$
with 
\begin{equation} \label{omega2general}
    \begin{pmatrix}
    0 & 0 & \cdots & 0 
    & \mathcal{S}_i (\nu, \lambda_2) h_1  & \dots & \mathcal{S}_i (\nu, \lambda_2) h_{n-m} \\
    \end{pmatrix}.
\end{equation}
Recalling our notation 
$\tilde{\omega}_1 (x; \lambda) := \omega_1 (f_1 (x; \lambda), \dots, f_n (x; \lambda))$,
we see that 
\begin{equation} \label{omega-tilde-prime}
   \tilde{\omega}_1' (x; \lambda)
   = \sum_{i = 1}^n D_i (x; \lambda, \lambda_2)
   = (\sum_{i=1}^n a_{i i} (x; \lambda)) \tilde{\omega}_1 (x; \lambda)
   + \sum_{i=1}^n \tilde{\omega}_2^i (x; \lambda, \lambda).
\end{equation}
At a crossing point $x_*$, $\tilde{\omega}_1 (x_*; \lambda) = 0$,
so that 
\begin{equation} \label{at-crossing}
    \tilde{\omega}_1' (x_*; \lambda) 
    = \sum_{i=1}^n \tilde{\omega}_2^i (x_*; \lambda, \lambda).
\end{equation}

Focusing now on the left shelf (i.e., $\lambda = \lambda_1$), 
in order to fix the sign of $\frac{\tilde{\omega}_1' (x_*; \lambda_1)}{\tilde{\omega}_2 (x_*; \lambda_1)}$,
we would like to choose $\omega_2$ based on the right-hand side of 
(\ref{at-crossing}) (with $\lambda = \lambda_1$), but we need to take care that $\omega_2$ 
is a properly defined skew-symmetric $n$-linear map. For this, 
we specify $\omega_2$ precisely as in (\ref{omega2-defined}), and we additionally 
set 
\begin{equation*}
    \tilde{\omega}_2 (x; \lambda) 
    := \omega_2 (f_1 (x; \lambda), \dots, f_n (x; \lambda)). 
\end{equation*}

We emphasize here the important point that $\omega_2$ has no explicit dependence 
on either $x$ or $\lambda$. Nonetheless, computing as above, except with 
$\mathbb{A}(\lambda, \lambda_2)$ replaced by $\tilde{\mathbb{A}}(\lambda_1, \lambda_2)$
(from (\ref{mathbb-A-tilde})),
we find that if $\{f_j (x; \lambda)\}_{j=1}^n$ are columns of the matrix $\mathbf{F} (x; \lambda)$ 
specified in (\ref{f-frame}) then (using Assumption {\bf (B)})
\begin{equation} \label{tilde-omega2}
    \tilde{\omega}_2 (x; \lambda)
    = \sum_{i=1}^n \tilde{\omega}_2^i (x; \lambda, \lambda_1).
\end{equation}
Combining this last relation with (\ref{at-crossing}), we see that 
with $\omega_1, \tilde{\omega}_1$, $\omega_2$, and $\tilde{\omega}_2$ as specified above, we have 
\begin{equation*}
    \frac{\tilde{\omega}_1' (x_*; \lambda_1)}{\tilde{\omega}_2 (x_*; \lambda_1)} = 1,
\end{equation*}
providing the claimed monotonicity. It follows from this monotonicity that the 
generalized Maslov index
$\ind (\mathpzc{g} (\cdot; \lambda_1), \mathpzc{h} (\cdot; \lambda_2); [0, 1])$
is a monotonic (positive) count of the number of times the subspaces $\mathpzc{g} (x; \lambda_1)$
and $\mathpzc{h} (x; \lambda_2)$ intersect (counted without multiplicity) as $x$ increases
from $0$ to $1$. This count can be expressed as 
\begin{equation*}
    \# \{x \in (0, 1]: \mathpzc{g} (x; \lambda_1) \cap \mathpzc{h} (x; \lambda_2) \ne \{0\} \},
\end{equation*}
where the omission of $x=0$ in the interval $(0,1]$ is because positively-oriented
crossing points don't increment the Maslov index on departures. This
completes the proof of Theorem \ref{main-theorem}.

\section{Invariance Framework}
\label{invariance-section}

Before turning to applications, we develop a framework for checking 
the invariance specified in Definition \ref{invariance-definition} 
(and assumed in the statement of Theorem \ref{main-theorem}). 
Here, we distinguish between invariance assumed along the boundary 
of the Maslov box (as in the statement of Theorem \ref{main-theorem})
and invariance throughout the interior of the Maslov box 
(which implies $\mathfrak{m} = 0$). This latter condition provides
substantially more information, and so it will be our primary 
focus. 

With the vector functions $\{f_j (x; \lambda)\}_{j=1}^n$ continuing 
to denote the columns of the frame $\mathbf{F} (x; \lambda)$ specified
in (\ref{f-frame}), we introduce the normalization factor 
\begin{equation} \label{f-normalization}
    d(x; \lambda) := |f_1 (x; \lambda) \wedge \dots \wedge f_n (x; \lambda)|
    = \sqrt{\det \mathbb{F} (x; \lambda)},
\end{equation}
where $\mathbb{F} (x; \lambda)$ denotes the Gram matrix with entries 
$(\mathbb{F} (x; \lambda))_{i j} = (f_i (x; \lambda), f_j (x; \lambda))$.
Due to the specific form of $\mathbf{F} (x; \lambda)$, we see that 
\begin{equation} \label{d-product}
    d(x; \lambda) = d_g (x; \lambda) d_h (x; \lambda_2),
\end{equation}
where 
\begin{equation} \label{gh-normalizations}
    \begin{aligned}
    d_g (x; \lambda) &:= |g_1 (x; \lambda) \wedge \dots \wedge g_m (x; \lambda)|
    = \sqrt{\det \mathbb{G} (x; \lambda)}, \\
    d_h (x; \lambda_2) &:= |h_1 (x; \lambda_2) \wedge \dots \wedge h_{n-m} (x; \lambda_2)|
    = \sqrt{\det \mathbb{H} (x; \lambda_2)},
    \end{aligned}
\end{equation}
with $\mathbb{G} (x; \lambda)$ denoting the Gram matrix with entries 
$(\mathbb{G} (x; \lambda))_{i j} = (g_i (x; \lambda), g_j (x; \lambda))$,
and $\mathbb{H} (x; \lambda_2)$ denoting the Gram matrix with entries 
$(\mathbb{H} (x; \lambda_2))_{i j} = (h_i (x; \lambda_2), h_j (x; \lambda_2))$. 
Since the elements $\{g_j (x; \lambda)\}_{j=1}^m$ are linearly independent, 
we have $d_g (x; \lambda) > 0$ for all $(x, \lambda) \in [0,1] \times [\lambda_1, \lambda_2]$,
and similarly for $d_h (x; \lambda_2)$ for all $x \in [0, 1]$. 
As a measure of how far these values remain bounded away from 0, we introduce the constants
\begin{equation} \label{lowercase-constants}
    \begin{aligned}
    c_g &:= \min_{\underset{\lambda \in [\lambda_1, \lambda_2]}{x \in [0, 1]}}
    \frac{d_g (x; \lambda)}{|g_1 (x; \lambda)| \cdots |g_m (x; \lambda)|} \\
    c_h &:= \min_{x \in [0, 1]}
    \frac{d_h (x; \lambda_2)}{|h_1 (x; \lambda_2)| \cdots |h_{n-m} (x; \lambda_2)|}. \\
    \end{aligned}    
\end{equation}
The values of $c_h$ can reasonably be obtained by computation, but we would generally like
to estimate $c_g$ by other means.

\begin{remark} \label{only-use}
For our applications, our point of view will be that the generalized Maslov index 
\begin{equation*}
\ind (\mathpzc{g} (\cdot; \lambda_1), \mathpzc{h} (\cdot; \lambda_2); [0, 1])    
\end{equation*}
is to be obtained by computation (possibly analytic, but more generally 
numerical), and so for most of this discussion 
we take $\mathbf{G} (x; \lambda_1)$ and $\mathbf{H} (x; \lambda_2)$
to be effectively known for all $x \in [0, 1]$. For invariance throughout 
the Maslov box, this
leaves the problem of understanding $\mathbf{G} (x; \lambda)$
for all $(x, \lambda) \in [0, 1] \times (\lambda_1, \lambda_2]$.
\end{remark}

Following the set-up in Section \ref{invariance-subsection}, we specify the normalized
functions 
\begin{equation} \label{normalized-variables}
    \psi_i (x; \lambda) := \frac{\tilde{\omega}_i (x; \lambda)}{d(x; \lambda)},
    \quad i = 1, 2.
\end{equation}
In order to establish invariance, we will set 
\begin{equation} \label{rho-defined-again}
    \rho (x; \lambda) := \frac{1}{2} (\psi_1 (x; \lambda)^2 + \psi_2 (x; \lambda)^2),    
\end{equation}
and our goal is to show that for all $(x,\lambda) \in [0,1] \times [\lambda_1, \lambda_2]$
we have $\rho (x; \lambda) > 0$. Following \cite{BCCJM2022}, our approach is to check
that $\rho (0; \lambda) > 0$, and to show that $\rho_x (x; \lambda)$ is bounded
below such that $\rho (x; \lambda)$ can never become 0. As noted in Remark 
\ref{only-use}, our aim is to use only values generated for the evaluation of 
$\ind (\mathpzc{g} (\cdot; \lambda_1), \mathpzc{h} (\cdot; \lambda_2); [0, 1])$;
i.e., values of $\mathbf{G} (x; \lambda_1)$ and $\mathbf{H} (x; \lambda_2)$ 
for all $x \in [0, 1]$. See Remark \ref{normalization-remark} below regarding the 
advantage of introducing the values $d (x; \lambda)$ for this part of the 
analysis. 

To start, we observe that, by construction, neither $\mathbf{G} (0; \lambda)$ nor 
$\mathbf{H} (0; \lambda_2)$ depends on $\lambda$, so $\rho (0; \lambda)$ is 
constant for all $\lambda \in [\lambda_1, \lambda_2]$. In particular,  
for all $\lambda \in [\lambda_1, \lambda_2]$, $\rho (0; \lambda)$ can be 
computed from the frames $\mathbf{G} (0; \lambda) = \mathbf{P}$ 
and $\mathbf{H} (0; \lambda_2)$, the 
latter of which will generally be obtained by computation. 

Turning to $\rho_x (x; \lambda)$, we can write 
\begin{equation*}
    \rho_x = \psi_1 \partial_x \psi_1 + \psi_2 \partial_x \psi_2,
\end{equation*}
from which we see that we need to understand 
\begin{equation} \label{psi-i-prime}
    \partial_x \psi_j (x; \lambda) 
    = \frac{\tilde{\omega}_j' (x; \lambda)}{d(x; \lambda)} - \frac{d'(x; \lambda)}{d(x; \lambda)} \psi_j (x; \lambda),
    \quad j = 1, 2.
\end{equation}
For $j=1$, we would like to use (\ref{omega-tilde-prime}) to
relate $\tilde{\omega}_1' (x; \lambda)$ to $\tilde{\omega}_2 (x; \lambda)$, but we must
take care in this, because the former includes a sum of values $\tilde{\omega}_2^i (x; \lambda, \lambda)$
and the latter a sum of values $\tilde{\omega}_2^i (x; \lambda, \lambda_1)$. We will 
see that in many important cases, including those arising from eigenvalue problems 
such as (\ref{general-form}), we have the straightforward relation 
\begin{equation} \label{straightforward-relation}
\tilde{\omega}_2^i (x; \lambda, \lambda)
= \frac{\lambda_2 - \lambda}{\lambda_2 - \lambda_1} \tilde{\omega}_2^i (x; \lambda, \lambda_1),
\end{equation}
for all $i = 1, 2, \dots, n$. In this case (i.e., when (\ref{straightforward-relation}) holds), 
we have the useful relation
(combining (\ref{omega-tilde-prime}) and (\ref{tilde-omega2}))
\begin{equation} \label{omega1prime-omega2}
    \tilde{\omega}_1' (x; \lambda) = \Big(\sum_{i=1}^n a_{ii} (x; \lambda) \Big) \tilde{\omega}_1 (x; \lambda)
    + \frac{\lambda_2 - \lambda}{\lambda_2 - \lambda_1} \tilde{\omega}_2 (x; \lambda).
\end{equation}
We will assume (\ref{straightforward-relation}) holds throughout this section (and it will hold for our 
applications). We note, however, that (\ref{straightforward-relation}) is not a requirement 
of Theorem \ref{main-theorem}, but rather characterizes a family of cases for which 
invariance is more readily verified. 

\begin{proposition} \label{invariance-proposition1}
For (\ref{nonhammy})-(\ref{bc}), let Assumptions {\bf (A)} and {\bf (B)} hold, and with 
$\{\tilde{\omega}^i_2 (x; \lambda, \nu)\}_{i=1}^n$ specified via (\ref{omega2general})
suppose (\ref{straightforward-relation}) holds. In addition, set 
\begin{equation} \label{ca-defined}
    C_a := \max_{x \in [0, 1]} \Big| \sum_{i=1}^n a_{i i} (x) \Big|,
\end{equation}
and let $C_d$, $\delta$, and $C$ be constants so that 
\begin{equation*}
    \begin{aligned}
    \max_{\underset{\lambda \in [\lambda_1, \lambda_2]}{x \in [0, 1]}}    
    \Big| \frac{d' (x; \lambda)}{d (x; \lambda)} \Big| &\le C_d \\
    \max_{\underset{\lambda \in [\lambda_1, \lambda_2]}{x \in [0, 1]}}   
    \Big| \frac{\partial_x \tilde{\omega}_2 (x; \lambda)}{d (x; \lambda)} \Big| &\le \delta \\
    C := 2C_d + \max\{2C_a, 1\} &+ 1.
    \end{aligned}
\end{equation*}
If 
\begin{equation} \label{invariance-criterion}
    \rho(0; \lambda) > \frac{\delta^2}{2C} (e^C - 1),
\end{equation}
then $\rho (x; \lambda) > 0$ for all 
$(x, \lambda) \in [0, 1] \times [\lambda_1, \lambda_2]$. In particular, 
the triple $(\mathpzc{f} (\cdot; \cdot), \omega_1, \omega_2)$ is 
invariant on $[0, 1] \times [\lambda_1, \lambda_2]$ in the 
sense of Definition \ref{invariance-definition}.
\end{proposition}

\begin{proof}
First, combining (\ref{psi-i-prime}) (with $j=1$) and (\ref{omega1prime-omega2}), we see that 
\begin{equation*}
    \psi_1 (x; \lambda) \partial_x \psi_1 (x; \lambda)
    = \Big( (\sum_{i=1}^n a_{i i} (x; \lambda)) - \frac{d' (x; \lambda)}{d (x; \lambda)} \Big) \psi_1 (x; \lambda)^2
    + \frac{\lambda_2 - \lambda}{\lambda_2 - \lambda_1} \psi_1 (x; \lambda) \psi_2 (x; \lambda),
\end{equation*}
and we can also write 
\begin{equation*}
    \psi_2 (x; \lambda) \partial_x \psi_2 (x; \lambda)
    = \psi_2 (x; \lambda) \frac{\tilde{\omega}_2' (x; \lambda)}{d(x; \lambda)}
    - \frac{d'(x; \lambda)}{d(x; \lambda)} \psi_2 (x; \lambda)^2.
\end{equation*}
Combining these observations, we arrive at the relation
\begin{equation} \label{psi-derivative}
    \rho' (x; \lambda) = (\sum_{i=1}^n a_{i i} (x; \lambda)) \psi_1^2
    - \frac{d'}{d} (\psi_1^2 + \psi_2^2) + \frac{\lambda_2 - \lambda}{\lambda_2 - \lambda_1} \psi_1 \psi_2 
    + \frac{\tilde{\omega}_2'}{d} \psi_2.
\end{equation}
Using the estimates 
\begin{equation*}
    \begin{aligned}
    \Big|\frac{\lambda_2 - \lambda}{\lambda_2 - \lambda_1} \psi_1 \psi_2 \Big| &\le \rho \\
    |\frac{\tilde{\omega}_2'}{d} \psi_2| &\le \frac{1}{2} (\delta^2 + \psi_2^2),
    \end{aligned}
\end{equation*}
both holding for all 
$(x, \lambda) \in [0, 1] \times [\lambda_1, \lambda_2]$,
along with the definitions of $C_a$ and $C_d$
we obtain the differential inequality 
\begin{equation}
    \begin{aligned}
    \rho' &\ge - C_a \psi_1^2 - (2C_d + 1) \rho - \frac{1}{2} \delta^2 - \frac{1}{2} \psi_2^2 \\
    & \ge - (2C_d + \max \{2C_a, 1 \} + 1) \rho - \frac{1}{2} \delta^2, 
    \end{aligned}
\end{equation}
which we can express as 
\begin{equation*}
    \rho' \ge - C \rho - \frac{1}{2} \delta^2.
\end{equation*}
Upon expressing this final inequality as 
$(e^{Cx} \rho)' \ge - \frac{1}{2} \delta^2 e^{Cx}$ and integrating 
both sides on $[0, x]$, we obtain the relation 
\begin{equation*}
    e^{Cx} \rho (x; \lambda) \ge \rho (0; \lambda) - \frac{\delta^2}{2C} (e^{Cx} - 1)
    \ge \rho (0; \lambda) - \frac{\delta^2}{2C} (e^{C} - 1).
\end{equation*}
The claim follows immediately. 
\end{proof}

We see from Proposition \ref{invariance-proposition1} that invariance can be 
established from the three values $\rho (0; \lambda)$, $C_d$, 
and $\delta$ (along with the easily obtained value $C_a$). 
We have already seen that the value of $\rho (0; \lambda)$
can be obtained in a natural way by computation of $\mathbf{H} (0; \lambda_2)$, 
so we turn next to the value $C_d$, for which we first observe from (\ref{d-product})
the relation
\begin{equation} \label{d-sum}
    \frac{d' (x; \lambda)}{d (x; \lambda)}
    = \frac{d_g' (x; \lambda)}{d_g (x; \lambda)}
    + \frac{d_h' (x; \lambda_2)}{d_h (x; \lambda_2)}.
\end{equation}
Taking a maximum on both sides of this relation leads to the inequality
\begin{equation*}
    \max_{\underset{\lambda \in [\lambda_1, \lambda_2]}{x \in [0, 1]}} \Big|  
    \frac{d' (x; \lambda)}{d (x; \lambda)} \Big| 
    \le \max_{\underset{\lambda \in [\lambda_1, \lambda_2]}{x \in [0,1]}} 
    \Big| \frac{d_g' (x; \lambda)}{d_g (x; \lambda)} \Big| 
    +  \max_{x \in [0, 1]} \Big| \frac{d_h' (x; \lambda_2)}{d_h (x; \lambda_2)} \Big|.
\end{equation*}
We have the following proposition. 

\begin{proposition} \label{invariance-proposition2}
For (\ref{nonhammy})-(\ref{bc}), let Assumptions {\bf (A)} and {\bf (B)} hold, and with 
$\{\tilde{\omega}^i_2 (x; \lambda, \nu)\}_{i=1}^n$ specified via (\ref{omega2general})
suppose (\ref{straightforward-relation}) holds. In addition, set 
\begin{equation} \label{cap-a-defined}
C_A := \max_{\underset{\lambda \in [\lambda_1, \lambda_2]}{x \in [0,1]}}
\|A (x; \lambda)\|.
\end{equation}
Then 
\begin{equation*}
    \begin{aligned}
    \max_{\underset{\lambda \in [\lambda_1, \lambda_2]}{x \in [0,1]}} 
    \Big| \frac{d_g' (x; \lambda)}{d_g (x; \lambda)} \Big| 
    &\le \frac{m! C_A}{c_g^2} \\
    \max_{x \in [0, 1]} \Big| \frac{d_h' (x; \lambda_2)}{d_h (x; \lambda_2)} \Big|
    &\le \frac{(n-m)!}{c_h^2} \max_{x \in [0, 1]} \|A (x; \lambda_2)\|,
    \end{aligned}
\end{equation*}
and also 
\begin{equation*}
    d (x; \lambda) \ge c_g c_h |g_1 (x; \lambda)| \cdots |g_m (x; \lambda)| 
    |h_1 (x; \lambda_2)| \cdots |h_{n-m} (x; \lambda_2)|, 
\end{equation*}
for all $(x, \lambda) \in [0, 1] \times [\lambda_1, \lambda_2]$. 
\end{proposition}

\begin{proof}
Beginning with $d_g (x; \lambda)$, we recall (\ref{gh-normalizations})
and use Jacobi's formula to compute
\begin{equation*}
    \begin{aligned}
    d_g' (x; \lambda) 
    &= \frac{1}{2 d_g (x; \lambda)} \frac{\partial}{\partial x} \det \mathbb{G} (x; \lambda) \\
    &= \frac{1}{2} d_g (x; \lambda) \tr (\mathbb{G} (x; \lambda)^{-1} \mathbb{G}' (x; \lambda)),
    \end{aligned}
\end{equation*}
from which we see that 
\begin{equation} \label{dg-prime}
    \frac{d_g' (x; \lambda)}{d_g (x; \lambda)}
    = \frac{1}{2} \tr (\mathbb{G} (x; \lambda)^{-1} \mathbb{G}' (x; \lambda)).
\end{equation}
Likewise, 
\begin{equation} \label{dh-prime}
    \frac{d_h' (x; \lambda_2)}{d_h (x; \lambda_2)}
    = \frac{1}{2} \tr (\mathbb{H} (x; \lambda_2)^{-1} \mathbb{H}' (x; \lambda_2)),
\end{equation}
and so our goal becomes to estimate values for the constants 
\begin{equation} \label{uppercase-constants}
    \begin{aligned}
    C_g &:= \max_{\underset{\lambda \in [\lambda_1, \lambda_2]}{x \in [0,1]}}
    \frac{1}{2} \Big| \tr (\mathbb{G} (x; \lambda)^{-1} \mathbb{G}' (x; \lambda)) \Big| \\
    C_h &:= \max_{x \in [0, 1]} 
    \frac{1}{2} \Big| \tr (\mathbb{H} (x; \lambda_2)^{-1} \mathbb{H}' (x; \lambda_2)) \Big|.
    \end{aligned}
\end{equation}
As with $c_h$, the value of $C_h$ can reasonably obtained by computation, 
but we would generally like to estimate $C_g$ by other means. 

Toward this end, we begin by recalling that 
$(\mathbb{G} (x; \lambda))_{i j} = (g_i (x; \lambda), g_j (x;\lambda))$,
and so 
\begin{equation*}
    (\mathbb{G}' (x; \lambda))_{i j} = (A (x; \lambda) g_i (x; \lambda), g_j (x;\lambda))
    + (g_i (x; \lambda), A (x; \lambda) g_j (x;\lambda)),
\end{equation*}
from which we see that for all $i, j \in \{1, 2, \dots, n\}$
\begin{equation} \label{gram-prime-estimates}
    |(\mathbb{G}' (x; \lambda))_{i j}| \le 2 \|A (x; \lambda)\| |g_i (x; \lambda)| |g_j (x; \lambda)|,
\end{equation}
for all $(x, \lambda) \in [0,1] \times [\lambda_1, \lambda_2]$. Next, if we let 
$M (x; \lambda) = (m_{i j} (x; \lambda))$ denote the adjugate matrix for
$\mathbb{G} (x; \lambda)$, then $\mathbb{G} (x; \lambda)^{-1} = \frac{M(x; \lambda)}{d_g (x; \lambda)^2}$,
and we can bound the entries of $M$ as follows: for any collection of distinct 
indices $\{i_{k}\}_{k=1}^m$ 
\begin{equation} \label{adjugate-estimates}
    \begin{aligned}
    |m_{i_1 i_1} (x; \lambda)| &\le (m-1)! |g_{i_2} (x; \lambda)|^2 \cdots |g_{i_m} (x; \lambda)|^2 \\
     |m_{i_1 i_2} (x; \lambda)| &\le (m-1)! |g_{i_1} (x; \lambda)| |g_{i_2} (x; \lambda)| |g_{i_3} (x; \lambda)|^2 \cdots |g_{i_m} (x; \lambda)|^2.
    \end{aligned}
\end{equation}

We can compute 
\begin{equation*}
    \begin{aligned}
    (\mathbb{G} (x; \lambda))^{-1} \mathbb{G}' (x; \lambda))_{i i} 
    &= \sum_{k=1}^m ((\mathbb{G} (x; \lambda))^{-1})_{i k} (\mathbb{G}' (x; \lambda))_{k i} \\
    &= \frac{1}{d_g (x; \lambda)^2} 
    \sum_{k=1}^m m_{i k} (x; \lambda) (\mathbb{G}' (x; \lambda))_{k i},
    \end{aligned}
\end{equation*}
and combining (\ref{gram-prime-estimates}) and (\ref{adjugate-estimates}) we can conclude 
that for all $i, k \in \{1, 2, \dots, m\}$
\begin{equation*}
    |m_{i k} (x; \lambda) (\mathbb{G}' (x; \lambda))_{k i}|
    \le 2 (m-1)! \|A (x; \lambda)\| |g_1 (x; \lambda)|^2 \cdots |g_m (x; \lambda)|^2.
\end{equation*}
In this way, we see that  
\begin{equation}
\begin{aligned}
    |\frac{d_g' (x; \lambda)}{d_g (x; \lambda)}|
    &\le \frac{(m-1)!}{d_g (x; \lambda)^2} \sum_{k=1}^m \|A (x; \lambda)\| |g_1 (x; \lambda)|^2 \cdots |g_n (x; \lambda)|^2 \\
    &= \frac{m!}{d_g (x; \lambda)^2} \|A (x; \lambda)\| |g_1 (x; \lambda)|^2 \cdots |g_n (x; \lambda)|^2.
\end{aligned}
\end{equation}
According to the specification of $c_g$ in (\ref{lowercase-constants}), we obtain the estimate 
\begin{equation*}
    |\frac{d_g' (x; \lambda)}{d_g (x; \lambda)}|
    \le \frac{m! \|A (x; \lambda)\|}{c_g^2},
\end{equation*}
for all $(x, \lambda) \in [0, 1] \times [\lambda_1, \lambda_2]$, allowing us to write 
\begin{equation} \label{cg-inequality}
    C_g \le \frac{m! C_A}{c_g^2}.
\end{equation}

The estimate on $\frac{d_h' (x; \lambda_2)}{d_h (x; \lambda_2)}$ follows by 
an essentially identical calculation, and the final inequality in Proposition 
\ref{invariance-proposition2} is an immediate consequence of 
(\ref{d-product}) and (\ref{lowercase-constants}). 
\end{proof}

These considerations still leave the critical term 
$\frac{\tilde{\omega}_2' (x; \lambda)}{d (x; \lambda)}$ 
to be evaluated. In general, the evaluation of this ratio 
is quite cumbersome, so we will only analyze it in detail 
for the two specific classes of equations addressed in our section
on applications.

\section{Applications} 
\label{applications-section}

Our development, including monotonicity, is widely applicable to any 
system of form (\ref{nonhammy}) for which Assumptions {\bf (A)} and 
{\bf (B)} hold, with one substantial caveat: invariance is often problematic
to check. Nonetheless, we start with an important family of examples for which 
invariance is especially tractable.

\subsection{Single Higher Order Equations}
\label{higher-order-section}

In this section, we consider eigenvalue problems with the form 
\begin{equation} \label{higher-order-equation}
    (\alpha_n (x; \kappa_n) \phi^{(n-1)})' + \sum_{j=2}^{n-1} \alpha_j (x; \kappa_j) \phi^{(j)} 
    + \alpha_1 (x) \phi' + \alpha_0 (x) \phi = \lambda \phi, 
\end{equation}
$x \in (0, 1)$, $\phi (x; \lambda) \in \mathbb{R}$,
for some integer $n \ge 2$, and for which we assume 
$\alpha_0, \alpha_1 \in C ([0, 1], \mathbb{R})$,
$\{\alpha_j (\cdot; \kappa_j)\}_{j=2}^{n-1} \subset C ([0, 1], \mathbb{R})$,
and $\alpha_n (\cdot; \kappa_n) \in C^1 ([0, 1], \mathbb{R})$, with 
$\alpha_n (x; \kappa_n) \ge \alpha_n^0 > 0$ for all $x \in [0, 1]$ for some 
fixed value $\alpha_n^0$. Here, $\phi^{(j)}$ denotes the $j^{\rm th}$ 
derivative of $\phi$ with respect to $x$, and 
the non-zero parameters $\{\kappa_j\}_{j=2}^n$ have been introduced 
in anticipation of our discussion of invariance, and can be
viewed as fixed values for other parts of the discussion.
Generally, for each 
$j \in \{2, 3, \dots, n\}$, we view $\kappa_j$ as capturing the size
of the coefficient $\alpha_j (x; \kappa_j)$; often, we have in 
mind $\alpha_j (x; \kappa_j) = \kappa_j$ for at least some indices
$j \in \{2, \dots, n\}$. The analysis does not require flexibility 
in adjusting the sizes of $a_0 (x)$ and $a_1 (x)$, so no constants
are incorporated into those terms. 

Our interest in such equations is particularly 
motivated by the linearization of dispersive--diffusive PDE such 
as 
\begin{equation} \label{dispersive-diffusive}
    u_t + f(u)_x = (b(u)u_x)_x + (c (u) u_{xx})_x 
\end{equation}
about stationary solutions $\bar{u} (x)$, and similarly 
for fourth-order equations of generalized Cahn-Hilliard form 
\begin{equation} \label{generalized-ch}
u_t = (b(u)u_x)_x - (c(u) u_{xxx})_x 
\end{equation}
(primarily on unbounded domains in both cases). See, e.g., 
\cite{HZ2000} for a discussion of the former, \cite{Howard2007} 
for a discussion of the latter, and \cite{PW1992} for a broader 
view of the spectral analysis of nonlinear waves arising 
in single equations of higher order (via the Evans function rather
than the Maslov index).  

We express (\ref{higher-order-equation}) as a first order system 
by introducing a vector function $y \in \mathbb{R}^n$ with coordinates
$y_1 = \phi$, $y_2 = \kappa_2 \phi'$, $y_3 = \kappa_3 \phi''$, ..., 
$y_{n-1} = \kappa_{n-1} \phi^{(n-2)}$, $y_n = \alpha_n (x; \kappa_n) \phi^{(n-1)}$.
In this way, we obtain (\ref{nonhammy}) with 
\begin{equation} \label{higher-order-A}
    A (x; \lambda) 
    = \begin{pmatrix}
    0 & \frac{1}{\kappa_2} & 0 & 0 & \dots & 0 & 0 \\
    0 & 0 & \frac{\kappa_2}{\kappa_3} & 0 & \dots & 0 & 0 \\
     0 & 0 & 0 & \frac{\kappa_3}{\kappa_4} & \dots & 0 & 0 \\
    \vdots & \vdots & \vdots & \vdots & \vdots & \vdots & \vdots \\
    0 & 0 & 0 & 0 & \dots & 0 & \frac{\kappa_{n-1}}{\alpha_n (x; \kappa_n)} \\
    \lambda - \alpha_0 (x) & - \frac{\alpha_1 (x)}{\kappa_2} & - \frac{\alpha_2 (x; \kappa_2)}{\kappa_3} & - \frac{\alpha_3 (x; \kappa_3)}{\kappa_4} 
    & \dots & - \frac{\alpha_{n-2} (x; \kappa_{n-2})}{\kappa_{n-1}} & - \frac{\alpha_{n-1} (x; \kappa_{n-1})}{\alpha_n (x; \kappa_n)}
     \end{pmatrix},
\end{equation}
for which we immediately see that Assumption {\bf (A)} is satisfied. 
(Here, we recognize that expressions such as (\ref{higher-order-A}) are quite cumbersome, but in 
certain places they seem to provide greater clarity than their counterpart forms expressed
with more compact notation.)
In addition, it's clear by inspection that we have the relations
\begin{equation*}
    a_{i i} (x; \lambda) 
    = \begin{cases}
    0 & i \in \{1, 2, \dots, n - 1\} \\
    - \frac{\alpha_{n-1} (x; \kappa_{n-1})}{\alpha_n (x; \kappa_n)} & i = n,
    \end{cases} 
\end{equation*}
and 
\begin{equation*}
    a_{i j} (x; \lambda_2) - a_{i j} (x; \lambda) 
    = \begin{cases}
    \lambda_2 - \lambda & (i,j) = (n,1) \\
    0 & \textrm{otherwise},
    \end{cases}
\end{equation*}
and we can conclude that Assumption {\bf (B)} holds as well. It follows that we can apply
Theorem \ref{main-theorem} as long as we can check the invariance condition
of Definition \ref{invariance-definition}. Following our general 
discussion of invariance in Section \ref{invariance-section}, the main 
thing we have left to understand is the ratio 
$\frac{\tilde{\omega}_2' (x; \lambda)}{d (x; \lambda)}$.

In order to understand $\tilde{\omega}_2' (x; \lambda)$, we begin 
by observing from the definition of $ \mathcal{S}_i (\lambda, \lambda_2) h_j$
in (\ref{difference-sums}) that in this case 
\begin{equation*}
     \mathcal{S}_i (\lambda,\lambda_2) h_j 
     = \begin{cases}
     (\lambda_2 - \lambda) h_{1 j} & (i, j) \in \{n\} \times \{1, \dots, n - m\}, \\
     0 & \textrm{otherwise},
     \end{cases}
\end{equation*}
where $m$ is specified from the boundary conditions 
(\ref{bc}). It's now clear from (\ref{omega2general}) that
$\tilde{\omega}_2^i (x; \lambda, \lambda_1) \equiv 0$ for all 
$i \in \{1, 2, \dots, n-1\}$ so that (from (\ref{tilde-omega2}))
$\tilde{\omega}_2 (x; \lambda) = \tilde{\omega}_2^n (x; \lambda, \lambda_1)$,
where $\tilde{\omega}_2^n (x; \lambda, \lambda_1)$ is the determinant 
of the matrix obtained by replacing the final row of 
$(\mathbf{G} (x; \lambda) \,\, \mathbf{H} (x; \lambda_2))$ with 
\begin{equation*}
   \begin{pmatrix}
    0 & \dots & 0
    & (\lambda_2 - \lambda_1) h_{11} (\lambda_2) & \dots & (\lambda_2 - \lambda_1) h_{1 (n-m)} (\lambda_2)
    \end{pmatrix}.
\end{equation*}
With this characterization of $\tilde{\omega}_2^n (x; \lambda, \lambda_1)$
it's clear that condition (\ref{straightforward-relation}) holds. 

Upon differentiating this last determinant, we obtain a sum of $n$
determinants, each with a derivative on all the entries in 
exactly one row. It's straightforward to see that the 
first $n-2$ summands will be 0, leaving only the 
final two, namely 
\begin{equation} \label{higher-det1}
    \frac{(\lambda_2 - \lambda_1) \kappa_{n-1}}{\alpha_n (x; \kappa_n)}
    \det
    \begin{pmatrix}
    g_{11} (\lambda) & \dots & g_{1m} (\lambda) & h_{11} (\lambda_2) & \dots & h_{1(n-m)} (\lambda_2) \\
    g_{21} (\lambda) & \dots & g_{2m} (\lambda) & h_{21} (\lambda_2) & \dots & h_{2(n-m)} (\lambda_2) \\
    \vdots & \vdots & \vdots & \vdots & \vdots & \vdots \\
    g_{(n-2)1} (\lambda) & \dots & g_{(n-2)m} (\lambda) & h_{(n-2)1} (\lambda_2) & \dots & h_{(n-2)(n-m)} (\lambda_2) \\
    g_{n1} (\lambda) & \dots & g_{nm} (\lambda) & h_{n1} (\lambda_2) & \dots & h_{n(n-m)} (\lambda_2) \\
    0 & \dots & 0 & h_{11} (\lambda_2) & \dots & h_{1(n-m)} (\lambda_2) 
    \end{pmatrix},
\end{equation}
and
\begin{equation} \label{higher-det2}
    \frac{\lambda_2 - \lambda_1}{\kappa_2}
    \det
    \begin{pmatrix}
    g_{11} (\lambda) & \dots & g_{1m} (\lambda) & h_{11} (\lambda_2) & \dots & h_{1(n-m)} (\lambda_2) \\
    g_{21} (\lambda) & \dots & g_{2m} (\lambda) & h_{21} (\lambda_2) & \dots & h_{2(n-m)} (\lambda_2) \\
    \vdots & \vdots & \vdots & \vdots & \vdots & \vdots \\
    g_{(n-1)1} (\lambda) & \dots & g_{(n-1)m} (\lambda) & h_{(n-1)1} (\lambda_2) & \dots & h_{(n-1)(n-m)} (\lambda_2) \\
    0 & \dots & 0 & h_{21} (\lambda_2) & \dots & h_{2(n-m)} (\lambda_2)
    \end{pmatrix}.
\end{equation}

Applying Hadamard's inequality for the determinant of a matrix to each
of these last two determinants, we obtain the estimate 
\begin{equation} \label{higher-omega2-derivative}
    |\tilde{\omega}_2' (x; \lambda)|
    \le (\lambda_2 - \lambda_1) \{\frac{\kappa_{n-1}}{\alpha_n (x; \kappa_n)} + \frac{1}{\kappa_2}\} 
    |g_1 (x; \lambda)| \cdots |g_m (x; \lambda)|
    |h_1 (x; \lambda_2)| \cdots |h_{n-m} (x; \lambda_2)|,
\end{equation}
for all $(x, \lambda) \in [0,1] \times [\lambda_1, \lambda_2]$. 
Combining (\ref{higher-omega2-derivative}) with the final assertion of Proposition
\ref{invariance-proposition2}, 
we obtain the estimate 
\begin{equation} \label{higher-omega2-derivative2}
  |\frac{\tilde{\omega}_2' (x; \lambda)}{d (x; \lambda)}| 
  \le \frac{\lambda_2 - \lambda_1}{c_g c_h}
  \{\frac{\kappa_{n-1}}{\alpha_n (x; \kappa_n)} + \frac{1}{\kappa_2}\}. 
\end{equation}

\begin{remark} \label{normalization-remark} In the absence of 
normalization by $d(x; \lambda)$, we would need to obtain 
estimates directly on (\ref{higher-omega2-derivative}), which 
is problematic since we prefer to avoid computing the values 
of $\{g_i (x; \lambda)\}_{i=1}^m$. The use of normalization 
allows us to use the right-hand side of (\ref{higher-omega2-derivative2})
as our estimate. 
\end{remark}

\subsubsection{The Case $m = 1$}
\label{mequals1-section}

The case $m = 1$ is especially amenable to analysis, because in 
that case we have simply $\mathbb{G} (x; \lambda) = |g_1 (x; \lambda)|^2$,
from which it follows immediately from (\ref{lowercase-constants}) that 
$c_g = 1$, and (from (\ref{cg-inequality})) $C_g \le C_A$ (with 
$C_A$ as defined in Proposition \ref{invariance-proposition2}). 
Combining these observations, we 
see that in this case, the constants $C$ and $\delta$ from 
Proposition \ref{invariance-proposition1} can be taken to be 
\begin{equation} \label{mequals1}
    \begin{aligned}
    C &= 2 (C_A + C_h)
    + \max \{ 2 \max_{x \in [0, 1]} \Big|\frac{\alpha_{n-1} (x; \kappa_{n-1})}{\alpha_n (x; \kappa_n)}\Big|, 1 \} + 1 \\
    \delta &= \frac{\lambda_2 - \lambda_1}{c_h} \{ \max_{x \in [0, 1]} \frac{\kappa_{n-1}}{\alpha_n (x; \kappa_n)} + \frac{1}{\kappa_2}\}. 
    \end{aligned}
\end{equation}
Each of these values can be determined by computation along
the left shelf (see Section \ref{example1-section} for a 
detailed example case).

\subsubsection{The Case $m > 1$}
\label{m-greater1-section}

In the case $m > 1$, determination of the value 
$c_g$ becomes substantially more challenging. 
Nonetheless, we can make a general observation, adapted 
from \cite{BCCJM2022}. It's clear from (\ref{higher-omega2-derivative2}) 
that by taking  $\frac{\kappa_{n-1}}{\alpha_n (x; \kappa_n)}$ and $\frac{1}{\kappa_2}$ 
small, we can reduce 
$\delta$ as long as $c_g$ and $c_h$ remain uniformly bounded 
away from 0. As $\alpha_n (x; \kappa_n)$ becomes large relative to 
the other coefficients, (\ref{higher-order-equation})
is approximated by 
\begin{equation*}
    (\alpha_n (x; \kappa_n) \phi^{(n-1)})' = 0, 
\end{equation*}
allowing us to employ regular perturbation theory to show that indeed
$c_g$ and $c_h$ can be uniformly bounded away from 0. If, in addition,
$\rho (0; \lambda)$ remains uniformly bounded away from 0, we can conclude
invariance. We record the details of this observation in the following 
proposition, in which  
$\mathpzc{f} (x; \lambda)$ denotes the Grassmannian 
subspace with frame $\mathbf{F} (x; \lambda)$ specified
in (\ref{f-frame}), $\omega_1$ is specified in 
(\ref{omega1}), and $\omega_2$ is specified in 
(\ref{omega2-defined}).

\begin{proposition} \label{m-bigger-than-one-proposition}
Let $\lambda_1, \lambda_2 \in \mathbb{R}$, $\lambda_1 < \lambda_2$ be fixed.
In (\ref{higher-order-equation}), assume 
$\alpha_0, \alpha_1 \in C ([0, 1], \mathbb{R})$, and that for 
each $j \in \{2, 3, \dots, n\}$, 
$\alpha_j (x; \kappa_j) = \kappa_j \tilde{\alpha}_j (x; \kappa_j)$,
with $\tilde{\alpha}_j (\cdot; \kappa_j) \in C ([0, 1], \mathbb{R})$.
In addition, assume there exist constants $\{C_j\}_{j=2}^n$, along
with a constant $c_n > 0$, all independent of the values of 
$\{\kappa_j\}_{j=2}^n$, so that 
\begin{equation*}
    \max_{x \in [0, 1]} |\tilde{\alpha}_j (x; \kappa_j)| \le C_j
    \quad \forall \,\, j \in \{2, \dots, n\},
    \quad {\it and} \quad
    \max_{x \in [0, 1]} |\tilde{\alpha}_n (x; \kappa_n)| \ge c_n,
\end{equation*}
for all $\{\kappa_j\}_{j=2}^n$ for which 
\begin{equation}
    r := \max \{\frac{1}{\kappa_2}, \frac{\kappa_2}{\kappa_3}, \dots, \frac{\kappa_{n-1}}{\kappa_n}\}
\end{equation}
is sufficiently small. For boundary frames 
\begin{equation} \label{tri-frames}
    \mathbf{P} = 
    \begin{pmatrix}
    P_1 \\ \tilde{P} \\ P_n
    \end{pmatrix},
    \quad
    \mathbf{Q} = 
    \begin{pmatrix}
    Q_1 \\ \tilde{Q} \\ Q_n
    \end{pmatrix},
\end{equation}
with $P_1, P_n \in \mathbb{R}^{1 \times m}$, $\tilde{P} \in \mathbb{R}^{(n-2) \times m}$,
and likewise 
$Q_1, Q_n \in \mathbb{R}^{1 \times (n-m)}$, $\tilde{Q} \in \mathbb{R}^{(n-2) \times (n-m)}$,
suppose either 
\begin{equation} \label{bc-conditions1}
\det
\begin{pmatrix}
P_1 & Q_1 \\ \tilde{P} & \tilde{Q} \\P_n & Q_n - (\int_0^1 (\lambda_2 - \alpha_0 (\xi)) d \xi) Q_1
\end{pmatrix}    
\ne 0, \quad
{\it or} \quad
\det
\begin{pmatrix}
P_1 & Q_1 \\ \tilde{P} & \tilde{Q} \\0 & Q_1
\end{pmatrix}    
\ne 0.
\end{equation}
Then there exists a value $r_0 > 0$ sufficiently small so that for 
any values $\{\kappa_j\}_{j=2}^n$ for which $r \le r_0$
we have $\rho (x; \lambda) > 0$ for all $(x, \lambda) \in [0, 1] \times [\lambda_1, \lambda_2]$.
In particular,
the invariance condition specified in Definition \ref{invariance-definition}
is satisfied for the triple $(\mathpzc{f} (\cdot; \cdot), \omega_1, \omega_2)$
on $[0,1] \times [\lambda_1, \lambda_2]$, so $\mathfrak{m} = 0$ in 
Theorem \ref{main-theorem}. 
\end{proposition}

\begin{proof}
Under our assumptions, we can apply regular perturbation theory to 
see that with $A (x; \lambda)$ specified as in (\ref{higher-order-A}) 
solutions to (\ref{nonhammy}) will satisfy 
\begin{equation*}
    y(x; \lambda) = y_0 (x; \lambda) + \mathbf{O} (r), 
\end{equation*}
where $y_0 (x; \lambda)$ solves the system $y_0' = A_0 (x; \lambda) y_0$, 
with $A_0 (x; \lambda)$ the $n \times n$ matrix with only a single non-zero entry, 
$a_{n1} (x; \lambda) = \lambda - a_0 (x)$ and $\mathbf{O} (r)$ uniform for 
$x \in [0,1]$. If we express a generic 
initial vector as $p = (p_1, \tilde{p}, p_n)^T$ with $p_1 \in \mathbb{R}$,
$\tilde{p} \in \mathbb{R}^{n-2}$, and $p_n \in \mathbb{R}$, and solve 
$y_0' = A_0 (x; \lambda) y_0$ subject to $y_0 (0) = p$,
we find $y(x; \lambda) = (p_1, \tilde{p}, p_n + p_1 \int_0^x (\lambda - a_0 (\xi)) d\xi)$.
Using this, and proceeding similarly for $y_0' = A_0 (x; \lambda_2) y_0$ 
initialized at $x = 1$ with $y_0 (1) = q = (q_1, \tilde{q}, q_n)^T$, 
we find that our frames $\mathbf{G} (x; \lambda)$
and $\mathbf{H} (x; \lambda_2)$ specified respectively in (\ref{g-frame})
and (\ref{h-frame}) satisfy the relations 
\begin{equation}
\begin{aligned}
    \mathbf{G} (x; \lambda)
    &= \begin{pmatrix}
    P_1 \\ \tilde{P} \\ P_n + (\int_0^x (\lambda - a_0 (\xi)) d \xi) P_1
\end{pmatrix} + \mathbf{O} (r), \\
    \mathbf{H} (x; \lambda_2)
    &= \begin{pmatrix}
    Q_1 \\ \tilde{Q} \\ Q_n - (\int_x^1 (\lambda_2 - a_0 (\xi)) d \xi) Q_1
\end{pmatrix} + \mathbf{O} (r).
\end{aligned}
\end{equation}

Since the lowest order frames are independent of $r$, we see 
that the constants $c_g$ and $c_h$ specified in (\ref{lowercase-constants})
can be bounded below for $r$ sufficiently small by positive constants
independent of the values $\{\kappa_j\}_{j=2}^n$. With this 
observation, along with (\ref{higher-omega2-derivative2}), we see
that we can make $\delta$ as small as we like by choosing $r_0$ 
sufficiently small. In addition, using the estimates from Proposition
\ref{invariance-proposition2}, we see that the value of the constant 
$C$ in Proposition \ref{invariance-proposition1} can be bounded above, independently 
of $r$ (as long as $r \le r_0$). In order to conclude that (\ref{invariance-criterion})
from Proposition \ref{invariance-proposition1} holds, we need only show that 
$\rho (0; \lambda)$ can be bounded below, again independently of 
$r$. For this, we can write 
\begin{equation*}
    \tilde{\omega}_1 (0; \lambda)
    = \det
    \begin{pmatrix}
    P_1 & Q_1 \\ \tilde{P} & \tilde{Q} \\P_n & Q_n - (\int_0^1 (\lambda_2 - a_0 (\xi)) d \xi) Q_1
\end{pmatrix}  
+ \mathbf{O} (r),
\end{equation*}
and 
\begin{equation*}
    \tilde{\omega}_2 (0; \lambda)
    = (\lambda_2 - \lambda_1) \det
    \begin{pmatrix}
    P_1 & Q_1 \\ \tilde{P} & \tilde{Q} \\ 0 & Q_1
\end{pmatrix}  
+ \mathbf{O} (r).
\end{equation*}
The conditions stated in the proposition are precisely that 
at least one of these determinants is non-zero, ensuring that 
$\tilde{\omega}_1 (0; \lambda)^2 + \tilde{\omega}_2 (0; \lambda)^2 > 0$.
In addition, since the columns of the lowest order 
matrices in $\mathbf{G} (0; \lambda)$ and $\mathbf{H} (0; \lambda_2)$
are necessarily linearly independent and independent of 
the values $\{\kappa_j\}_{j=2}^n$, we can conclude that 
the values $d_g (0; \lambda)$ and $d_h (0; \lambda_2)$ are 
both bounded below independently of the values $\{\kappa_j\}_{j=2}^n$. 
The necessary bound below on $\rho (0; \lambda)$ follows, and this 
completes the proof. 
\end{proof}

\begin{remark}
Condition (\ref{bc-conditions1}) in Proposition \ref{m-bigger-than-one-proposition}
is easily seen to hold in many important cases. As a specific family of examples, 
suppose $n$ is even and the boundary frames are $\mathbf{P} = {0 \choose I}$ 
and $\mathbf{Q} = {I \choose \Phi}$ for some $\frac{n}{2} \times \frac{n}{2}$
matrix $\Phi$. Then 
\begin{equation} \label{remark-eq}
    \begin{pmatrix}
P_1 & Q_1 \\ \tilde{P} & \tilde{Q} \\P_n & Q_n - \int_0^1 (\lambda_2 - a_0 (\xi)) d \xi Q_1
\end{pmatrix}    
= \begin{pmatrix}
0 & I \\
I & \cdots
\end{pmatrix},
\end{equation}
where the dots indicate that the lower right $\frac{n}{2} \times \frac{n}{2}$ matrix
is irrelevant for this calculation. Since the determinant of the right-hand 
side of (\ref{remark-eq}) is non-zero, condition (\ref{bc-conditions1})
is satisfied in this case. On the other hand, it's clear that if 
the boundary frames $\mathbf{P}$ and $\mathbf{Q}$ are both Dirichlet 
(i.e., $\mathbf{P} = \mathbf{Q} = {0 \choose I}$), then both determinants in 
(\ref{bc-conditions1}) are 0, and the condition is not satisfied. 
\end{remark}

\subsubsection{Example Case}
\label{example1-section}


As a specific example case, we consider the single third-order equation 
\begin{equation} \label{third-order}
    \alpha_3 \phi''' + \alpha_2 \phi'' + \alpha_1 (x) \phi' + \alpha_0 (x) \phi = \lambda \phi,
\end{equation}
with coefficient values
\begin{equation} \label{a0-a3}
    \alpha_0 (x) = .2 \cos (10x) - .5 \cos (x/10); 
    \quad \alpha_1 (x) = 2 \sin (5x); 
    \quad \alpha_2 = 10; 
    \quad \alpha_3 = 60, 
\end{equation}
and boundary conditions 
\begin{equation*}
    \phi' (0) = 0; \quad \phi'' (0) = 0; \quad \phi''(1) = 0. 
\end{equation*}
(This example is purely for purposes of illustration and doesn't correspond
with any particular physical problem.)
In this case, it's natural to take $\kappa_i = \alpha_i$,
$i = 2,3$,
and we see from (\ref{higher-order-A}) that 
\begin{equation*}
    A(x; \lambda) =
    \begin{pmatrix}
    0 & \frac{1}{\alpha_2} & 0 \\
    0 & 0 & \frac{\alpha_2}{\alpha_3} \\
    \lambda - \alpha_0 (x) & - \frac{\alpha_1 (x)}{\alpha_2} & - \frac{\alpha_2}{\alpha_3}
    \end{pmatrix}.
\end{equation*}

Referring to our general framework, this corresponds with the case 
$m = 1$, and we can take the frames for $\mathpzc{p}$ 
and $\mathpzc{q}$ to respectively be 
\begin{equation} \label{example1-initial-frames}
    \mathbf{P}
    = \begin{pmatrix}
    1 \\ 0 \\ 0
    \end{pmatrix}
    \quad {\rm and} \quad 
     \mathbf{Q}
    = \begin{pmatrix}
    1 & 0 \\ 0 & 1 \\ 0 & 0
    \end{pmatrix}.
\end{equation}
We search for eigenvalues on the interval $[\lambda_1, \lambda_2] = [-1, 0]$.

In order to check the invariance condition of Lemma 
\ref{invariance-proposition1}, we compute $C$ and $\delta$ using 
(\ref{mequals1}), along with 
$\rho (0; \lambda)$. For this, we need values for $C_A$ (from 
(\ref{cap-a-defined})), $C_h$ (from (\ref{uppercase-constants})),
and $c_h$ (from (\ref{lowercase-constants})). The value $C_A$
can be determined directly (i.e., without solving (\ref{third-order})), 
and we find $C_A = .7481$. The 
values $C_h$ and $c_h$ are both computed by numerical evaluation 
of the frame $\mathbf{H} (x; \lambda_2)$, and we find 
$C_h = .2621$ and $c_h = .9975$. With these values, we compute 
\begin{equation*}
        C = 2 (C_A + C_h) + \max\{2 \frac{\alpha_2}{\alpha_3},1\} + 1 
        = 2 (.7481 + .2621) + 1 + 1 = 4.0202,
\end{equation*}
and 
\begin{equation*}
    \delta = \frac{\lambda_2 - \lambda_1}{c_h} \Big{\{} \frac{1}{6} + \frac{1}{10} \Big{\}}
    = \frac{1}{.9975} \cdot \frac{4}{15} = .2673.
\end{equation*}

We evaluate $\rho (0; \lambda)$ from the exact frame $\mathbf{G} (0; \lambda)$
and the numerically generated frame $\mathbf{H} (0; \lambda_2)$, and we 
find $\rho (0; \lambda) = .5000$. It follows that 
\begin{equation*}
    \rho (0; \lambda) - \frac{\delta^2}{2C} (e^C - 1) = .0136 > 0,
\end{equation*}
verifying that our invariance criterion is satisfied. 

We are now justified in using Theorem \ref{main-theorem} with $\mathfrak{m} = 0$ 
to compute a lower bound on the number of eigenvalues that (\ref{third-order})
has on the interval $[-1, 0]$. We proceed by numerically computing 
the generalized Maslov index $\ind (\mathpzc{g} (\cdot; -1), \mathpzc{h} (\cdot; 0); [0, 1])$.
The flow is necessarily monotonic, and we find a single crossing 
point at about $x = .535$ (with a stepsize in the computation of $.005$). 
We can conclude that (\ref{third-order}) has at least one eigenvalue 
on the interval $[-1, 0]$. Although this conclusion requires only a 
computation along the left shelf, the entire Maslov box for this example
is depicted in Figure \ref{eg1_figure}. In this case, we see that 
(\ref{third-order}) has only a single eigenvalue on the interval
$[-1, 0]$, located at about $\lambda = -.513$ (with a 
stepsize in the computation of $.001$). 

\begin{figure}[ht] 
\begin{center}\includegraphics[%
  width=12cm,
  height=8cm]{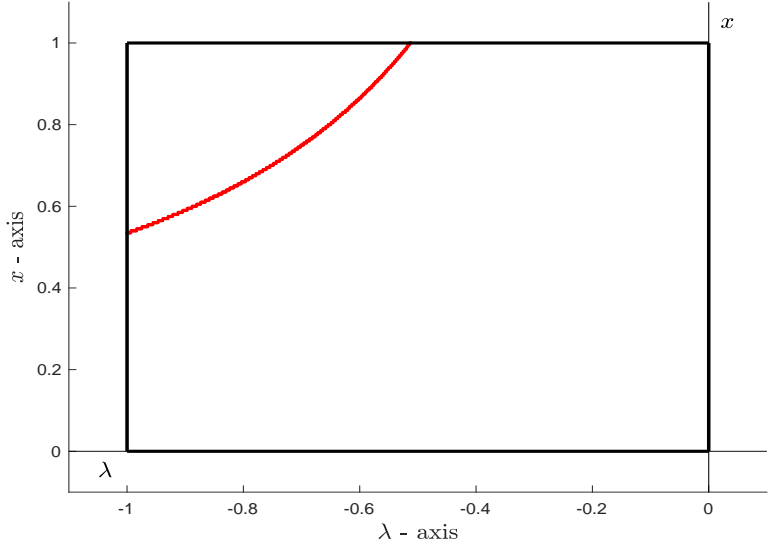}
\end{center}
\caption{Full Maslov Box and spectral curve for (\ref{third-order}).  
\label{eg1_figure}}
\end{figure}

\subsection{Second-Order Systems} 
\label{second-order-section}

In this section, we consider eigenvalue problems of the general form 
\begin{equation} \label{second-order-equation}
   -B \phi '' + W(x) \phi' + V(x) \phi = \lambda \phi,
   \quad x \in (0, 1), \quad \phi (x; \lambda) \in \mathbb{R}^l,
   \quad l \in \mathbb{N},
\end{equation}
for which we take $W, V \in C([0,1],\mathbb{R}^{l \times l})$
and assume for simplicity of the invariance verification 
that $B$ is a constant diagonal matrix
with positive diagonal entries $\{b_i\}_{i=1}^l$. As noted in the 
introduction, such equations arise naturally when 
we linearize a viscous conservation law (\ref{vcl-eqn}) 
about a viscous profile $\bar{u} (x - st)$. 

In order to place this system in the setting of 
(\ref{nonhammy}), we write $y = {y_1 \choose y_2}$
with $y_1 = \phi$ and $y_2 = B \phi'$, giving 
(\ref{nonhammy}) with $n = 2l$ and
\begin{equation} \label{second-order-A}
    A(x; \lambda)
    = \begin{pmatrix}
    0 & B^{-1} \\
    V(x) - \lambda I & W(x) B^{-1}
    \end{pmatrix},
\end{equation}
from which it's clear that our Assumption {\bf (A)} holds
in this case. Computing directly, we see that 
\begin{equation*}
    a_{i j} (x; \lambda_2) - a_{i j} (x; \lambda)
    = \begin{cases}
    - (\lambda_2 - \lambda) & (i, j) = (l+k,k), \,\, k \in \{1, \dots, l\} \\
    0 & \textrm{otherwise},
    \end{cases}
\end{equation*}
allowing us to conclude that {\bf (B)} holds as well. It follows that we can 
apply Theorem \ref{main-theorem} as long as we can verify the 
invariance condition specified in Definition \ref{invariance-definition}.

Following our general development, we fix any $m \in \{1, 2, \dots, 2l - 1\}$ 
and let 
$\mathbf{G} (x; \lambda) \in \mathbb{R}^{2l \times m}$ 
and $\mathbf{H} (x; \lambda) \in \mathbb{R}^{2l \times (2l-m)}$
be as specified respectively in (\ref{g-frame})
and (\ref{h-frame}). Then 
\begin{equation*}
    \tilde{\omega}_1 (x; \lambda) 
    = \det (\mathbf{G} (x; \lambda) \,\,\, \mathbf{H} (x; \lambda_2)),
\end{equation*}
and 
\begin{equation*}
    \tilde{\omega}_2 (x; \lambda)
    = \sum_{i=1}^{2l} \tilde{\omega}^i_2 (x; \lambda, \lambda_1),
\end{equation*}
where the functions $\{\tilde{\omega}^i_2 (x; \lambda, \lambda_1)\}_{i=1}^{2l}$ 
are as in (\ref{omega2general}). 

For invariance, we will focus on the case $m = l$, for which 
the boundary spaces $\mathpzc{p}$ and $\mathpzc{q}$ from 
(\ref{bc}) both have the same dimension $l$, and we will consider 
two cases of boundary conditions. For this we will refer to 
$\mathpzc{p}$ or $\mathpzc{q}$ as a Robin space if it has 
a frame of the form ${I \choose \Phi}$ for some $l \times l$
matrix $\Phi$. 

In the following proposition,
$\mathpzc{f} (x; \lambda)$ denotes the Grassmannian 
subspace with frame $\mathbf{F} (x; \lambda)$ specified
in (\ref{f-frame}), $\omega_1$ is specified in 
(\ref{omega1}), and $\omega_2$ is specified in 
(\ref{omega2-defined}).

\begin{proposition} \label{proposition2}
In (\ref{second-order-equation}), assume $W, V \in C([0,1],\mathbb{R}^{l \times l})$
and that $B$ is a constant diagonal matrix
with positive diagonal entries $\{b_i\}_{i=1}^l$. For boundary 
spaces $\mathpzc{p}$ and $\mathpzc{q}$ as specified in 
(\ref{bc}), suppose $\mathpzc{p}$ is Dirichlet and 
$\mathpzc{q}$ is Robin, or alternatively suppose 
$\mathpzc{q}$ is Dirichlet and 
$\mathpzc{p}$ is Robin. Then there exists a value 
$r_0 > 0$ sufficiently small so that for any values 
$\{b_i\}_{i=1}^l$ satisfying 
\begin{equation*}
    r := \max \{\frac{1}{b_1}, \frac{1}{b_2}, \dots, \frac{1}{b_n}\} \le r_0,
\end{equation*}
we have $\rho (x; \lambda) > 0$ for all 
$(x, \lambda) \in [0, 1] \times [\lambda_1, \lambda_2]$. In particular,
the invariance condition specified in Definition \ref{invariance-definition}
is satisfied for the triple $(\mathpzc{f} (\cdot; \cdot), \omega_1, \omega_2)$
on $[0,1] \times [\lambda_1, \lambda_2]$, so $\mathfrak{m} = 0$ in 
Theorem \ref{main-theorem}.  
\end{proposition}

\begin{proof}
Since the analysis is similar for each case, we carry out details 
only for the case in which (\ref{second-order-equation})
has Dirichlet boundary conditions at $x = 1$ and Robin boundary 
conditions at $x = 0$. 

Following the general development of Section \ref{invariance-section}, we 
see immediately that the values $C_a$ and $C_A$ can be bounded
independently of the values $\{b_i\}_{i=1}^l$ (for $r \le r_0$). In order to 
apply our general framework, we additionally need to verify 
that the values $c_g$ and $c_h$ specified in (\ref{lowercase-constants})
are bounded below uniformly as the values $\{b_i\}_{i=1}^l$
grow, and that by choosing the values $\{b_i\}_{i=1}^l$
sufficiently large we can make $\frac{\tilde{\omega}_2' (x; \lambda)}{d(x; \lambda)}$
as small as we like (without increasing the value of $C$). 

Beginning with the values $c_g$ and $c_h$, we notice that by 
regular perturbation theory for large values of $\{b_i\}_{i=1}^l$ 
the lowest order expression in a perturbation expansion for 
solutions of (\ref{nonhammy}) with (\ref{second-order-A}) solves 
the equation
\begin{equation*}
    y_0' = A_0 (x; \lambda) y_0,
    \quad 
    A_0 (x; \lambda)
    = 
    \begin{pmatrix}
        0 & 0 \\
        V(x) - \lambda I & 0
    \end{pmatrix}.
\end{equation*}
For $\mathbf{G} (x; \lambda)$, we take the boundary condition
$\mathbf{G} (0; \lambda) = {I \choose \Theta}$, and if we 
let $\mathbf{G}_0 (x; \lambda) = {G_0 (x; \lambda) \choose G_{00} (x; \lambda)}$
denote the lowest order term in a perturbation expansion for 
$\mathbf{G} (x; \lambda)$, then 
\begin{equation*} 
    G_0' (x; \lambda) = 0; \quad G_0 (0; \lambda) = I; 
    \quad \quad G_{00}' (x; \lambda) = (V(x) - \lambda I) G_0 (x; \lambda); 
    \quad G_{00} (0; \lambda) = \Theta.
\end{equation*}
Solving this system by integration, we conclude that 
\begin{equation} \label{lowest-order-G}
    \mathbf{G}_0 (x; \lambda)
    = \begin{pmatrix}
    I \\
    \mathcal{G} (x; \lambda)
    \end{pmatrix},
    \quad 
    \mathcal{G} (x; \lambda) = \Theta + \int_0^x (V(\xi) - \lambda I) d\xi,
\end{equation}
for all $x \in [0, 1]$. 

Likewise, for $\mathbf{H} (x; \lambda_2)$, we take the boundary condition
$\mathbf{H} (1; \lambda_2) = {0 \choose I}$, and if we 
let $\mathbf{H}_0 (x; \lambda_2) = {H_0 (x; \lambda_2) \choose H_{00} (x; \lambda_2)}$
denote the lowest order term in a perturbation expansion for 
$\mathbf{H} (x; \lambda_2)$, then 
\begin{equation*} 
    H_0' (x; \lambda_2) = 0; \quad H_0 (0; \lambda_2) = 0; 
    \quad \quad H_{00}' (x; \lambda_2) = (V(x) - \lambda_2 I) H_0 (x; \lambda_2); 
    \quad H_{00} (0; \lambda_2) = I.
\end{equation*}
Solving this system by integration, we conclude that 
\begin{equation} \label{lowest-order-H}
    \mathbf{H}_0 (x; \lambda_2)
    = \begin{pmatrix}
    0 \\
    I
    \end{pmatrix},
\end{equation}
for all $x \in [0, 1]$. 

Similarly as in the proof of Proposition \ref{m-bigger-than-one-proposition}, 
we can conclude that the values $c_g$ and $c_h$, viewed as functions 
of the values $\{b_i\}_{i=1}^l$, can be bounded below by 
positive constants that are independent of the value $r$ specified 
in Proposition \ref{proposition2} (as long as $r \le r_0$), 
and likewise we can conclude
that there exists a value $d_{\min} > 0$, independent 
of the values $\{b_i\}_{i=1}^n$, so that 
\begin{equation*}
    d(x; \lambda) \ge d_{\min},
    \quad \forall \,\, (x, \lambda, r) \in [0, 1] \times [\lambda_1, \lambda_2] \times [0, r_0]. 
\end{equation*}

Turning now to the ratio $\frac{\tilde{\omega}_2' (x; \lambda)}{d (x; \lambda)}$,
we first observe that in this case, 
\begin{equation} \label{mathcal-S-second-order}
    \mathcal{S}_i (\lambda,\lambda_2) h_j 
    := \begin{cases}
    - (\lambda_2 - \lambda) h_{i - l, j} & (i, j) \in \{l+1, \dots, 2l\} \times \{1, \dots, l\} \\
    0 & \textrm{otherwise},
    \end{cases}
\end{equation}
from which we immediately see from (\ref{omega2general}) that
$\tilde{\omega}_2^i (x; \lambda, \lambda_1) \equiv 0$ 
for all $i \in \{1, 2, \dots, l\}$. In order to understand the remaining
functions $\{\tilde{\omega}^i_2 (x; \lambda, \lambda_1)\}_{i=l+1}^{2l}$ 
in this case, we focus on $i = l + 1$ for which 
(from (\ref{omega2general})) $\tilde{\omega}_2^{l+1} (x; \lambda, \lambda_1)$
is the determinant of the matrix obtained by replacing the 
$(l+1)^{\textrm{th}}$ row of $(\mathbf{G} (x; \lambda) \,\, \mathbf{H} (x; \lambda_2))$
with 
\begin{equation*}
    \begin{aligned}
    &\begin{pmatrix}
    0 & \dots & 0
    & \mathcal{S}_{l+1} (\lambda_1, \lambda_2) h_1 (\lambda_2)  & \dots & \mathcal{S}_{l+1} (\lambda_1, \lambda_2) h_l (\lambda_2) \\
    \end{pmatrix} \\
    & = 
    \begin{pmatrix}
    0 & \dots & 0
    & - (\lambda_2 - \lambda_1) h_{11} (x; \lambda_2)  & \dots & - (\lambda_2 - \lambda_1) h_{1l} (x; \lambda_2) 
    \end{pmatrix}.
    \end{aligned}
\end{equation*}
From this relation, and similar relations for $\{\tilde{\omega}_2^i (x; \lambda, \lambda_1)\}_{i = l + 2}^{2l}$
and $\{\tilde{\omega}_2^i (x; \lambda, \lambda)\}_{i = l + 1}^{2l}$, we see that 
(\ref{straightforward-relation}) is satisfied. 

As in previous calculations along these lines, we compute 
the derivative of $\tilde{\omega}_2^{l+1} (x; \lambda, \lambda_1)$
as the sum of $2l$ determinants, each with a derivative on 
each entry in exactly one row. The first of these determinants is 
\begin{equation} \label{omega2-second-ellp1-prime1}
    \det
    \begin{pmatrix}
    g_{11}' (\lambda) & \dots & g_{1l}' (\lambda) 
    & h_{11}' (\lambda_2) & \dots & h_{1l}' (\lambda_2) \\
     \vdots & \vdots & \vdots & \vdots & \vdots & \vdots \\
    g_{l1} (\lambda) & \dots & g_{ll} (\lambda) 
    & h_{l1} (\lambda_2) & \dots & h_{l l} (\lambda_2) \\
    0 & \dots & 0
    & - (\lambda_2 - \lambda_1) h_{11} (\lambda_2)  & \dots & - (\lambda_2 - \lambda_1) h_{1l} (\lambda_2) \\
    g_{(l+2)1} (\lambda) & \dots & g_{(l+2)l} (\lambda) 
    & h_{(l+2)1} (\lambda_2) & \dots & h_{(l+2) l} (\lambda_2) \\ 
     \vdots & \vdots & \vdots & \vdots & \vdots & \vdots \\
    g_{(2l)1} (\lambda) & \dots & g_{(2l)l} (\lambda) 
    & h_{(2l)1} (\lambda_2) & \dots & h_{(2l)l} (\lambda_2)
    \end{pmatrix}.
\end{equation}
In this case, 
\begin{equation*}
    g_{1j}' = \frac{1}{b_1} g_{(l+1)j}, 
    \quad  h_{1j}' = \frac{1}{b_1} h_{(l+1)j},
    \quad \forall \, j \in \{1, 2, \dots, l\},
\end{equation*}
and we see that (\ref{omega2-second-ellp1-prime1}) becomes
\begin{equation} \label{omega2-second-ellp1-prime2}
    \frac{(\lambda_2 - \lambda_1)}{b_1} \det
    \begin{pmatrix}
    g_{(l+1)1} (\lambda) & \dots & g_{(l+1)l} (\lambda) 
    & h_{(l+1)1} (\lambda_2) & \dots & h_{(l+1) l} (\lambda_2) \\
     \vdots & \vdots & \vdots & \vdots & \vdots & \vdots \\
      g_{l1} (\lambda) & \dots & g_{ll} (\lambda) 
    & h_{l1} (\lambda_2) & \dots & h_{l l} (\lambda_2) \\
    0 & \dots & 0
    & - h_{11} (\lambda_2)  & \dots & - h_{1l} (\lambda_2) \\
    g_{(l+2)1} (\lambda) & \dots & g_{(l+2)l} (\lambda) 
    & h_{(l+2)1} (\lambda_2) & \dots & h_{(l+2) l} (\lambda_2) \\ 
     \vdots & \vdots & \vdots & \vdots & \vdots & \vdots \\
    g_{(2l)1} (\lambda) & \dots & g_{(2l)l} (\lambda) 
    & h_{(2l)1} (\lambda_2) & \dots & h_{(2l)l} (\lambda_2)
    \end{pmatrix}.
\end{equation}
Using Hadamard's inequality for determinants, we can bound this 
term by 
\begin{equation} \label{hadamard-b1}
\frac{\lambda_2 - \lambda_1}{b_1} |g_1 (x; \lambda)| \cdots |g_l (x; \lambda)|
|h_1 (x; \lambda_2)| \cdots |h_l (x; \lambda_2)|.
\end{equation}

For the next $(l-1)$ summands of $\partial_x \tilde{\omega}_2^{l+1} (x; \lambda, \lambda_1)$,
we similarly start with a derivative on the $j^{\rm th}$ row ($j \in \{2, \dots, l\}$) of
the matrix under determinant in $\tilde{\omega}_2^{l+1} (x; \lambda, \lambda_1)$.
In each of these cases, the $j^{\rm th}$ row becomes linearly dependent with the 
$(l+j)^{\rm th}$ row, and the resulting determinant is 0. This brings us to the 
summand obtained by differentiating the $(l+1)^{\rm st}$ row of
the matrix under determinant in $\tilde{\omega}_2^{l+1} (x; \lambda, \lambda_1)$,
and it's straightforward to see that this term can again be estimated by 
(\ref{hadamard-b1}). 

In order to understand the determinants with derivatives on rows $l+2$ through 
$2l$, we focus on the first. For this, we have 
\begin{equation*}
    \begin{aligned}
    \det &
    \begin{pmatrix}
    g_{11} (\lambda) & \dots & g_{1l} (\lambda) 
    & h_{11} (\lambda_2) & \dots & h_{1l} (\lambda_2) \\
     \vdots & \vdots & \vdots & \vdots & \vdots & \vdots \\
    g_{l1} (\lambda) & \dots & g_{ll} (\lambda) 
    & h_{l1} (\lambda_2) & \dots & h_{ll} (\lambda_2) \\ 
    0 & \dots & 0
    & - (\lambda_2 - \lambda_1) h_{11} (\lambda_2)  & \dots & - (\lambda_2 - \lambda_1) h_{1l} (\lambda_2) \\
    g_{(l+2)1}' (\lambda) & \dots & g_{(l+2)l}' (\lambda) 
    & h_{(l+2)1}' (\lambda_2) & \dots & h_{(l+2) l}' (\lambda_2) \\ 
     \vdots & \vdots & \vdots & \vdots & \vdots & \vdots \\
    g_{(2l)1} (\lambda) & \dots & g_{(2l)l} (\lambda) 
    & h_{(2l)1} (\lambda_2) & \dots & h_{(2l)l} (\lambda_2)
    \end{pmatrix} \\
    &= \det
    \begin{pmatrix}
    g_{11} (\lambda) & \dots & g_{1l} (\lambda) 
    & h_{11} (\lambda_2) & \dots & h_{1l} (\lambda_2) \\
     \vdots & \vdots & \vdots & \vdots & \vdots & \vdots \\
    g_{l1} (\lambda) & \dots & g_{ll} (\lambda) 
    & h_{l1} (\lambda_2) & \dots & h_{ll} (\lambda_2) \\ 
    0 & \dots & 0
    & - (\lambda_2 - \lambda_1) h_{11} (\lambda_2)  & \dots & - (\lambda_2 - \lambda_1) h_{1l} (\lambda_2) \\
    a_{(l+2)k} g_{k1} (\lambda) & \dots & a_{(l+2)k} g_{kl} (\lambda) 
    & a_{(l+2)k} h_{k1} (\lambda_2) & \dots & a_{(l+2)k} h_{kl} (\lambda_2) \\ 
     \vdots & \vdots & \vdots & \vdots & \vdots & \vdots \\
    g_{(2l)1} (\lambda) & \dots & g_{(2l)l} (\lambda) 
    & h_{(2l)1} (\lambda_2) & \dots & h_{(2l)l} (\lambda_2)
    \end{pmatrix}, 
    \end{aligned}
\end{equation*}
where for typesetting considerations we're using the convention of 
summing over any index appearing twice in an expression. E.g., 
written out in full
\begin{equation*}
    a_{(l+2)k} g_{k1} (\lambda)
    = \sum_{k=1}^{2l} a_{(l+2)k} (x; \lambda) g_{k1} (x; \lambda),
\end{equation*}
and similarly for other such entries. 

We can use row operations to eliminate all except two of the summands
involving components of $\mathbf{G} (x; \lambda)$ in 
row $(l+1)$. In particular, we can eliminate all summands {\it except} 
\begin{equation*}
    a_{(l+2)(l+1)} (x; \lambda) g_{(l+1)j} (x; \lambda)
    + a_{(l+2)(l+2)} (x; \lambda) g_{(l+2)j} (x; \lambda),
    \quad j = 1, 2, \dots, l.
\end{equation*}
For summands involving components of $\mathbf{H} (x; \lambda_2)$, we correspondingly 
obtain sums of the form 
\begin{equation*}
\begin{aligned}
    a_{(l+2)(l+1)} &(x; \lambda_2) h_{(l+1)j} (x; \lambda_2)
    + a_{(l+2)(l+2)} (x; \lambda_2) h_{(l+2)j} (x; \lambda_2) \\
    &+ \sum_{k \notin \{(l+1), (l+2)\}} (a_{(l+2)k} (x; \lambda_2) - a_{(l+2)k} (x; \lambda)) h_{k1} (x; \lambda_2),
    \quad j = 1, 2, \dots, l.
\end{aligned}
\end{equation*}

Using (\ref{second-order-A}), we see that
\begin{equation} \label{first-terms}
    a_{(l+2)(l+1)} (x; \lambda) = (W(x) B^{-1})_{21}
    = \frac{W_{21} (x)}{b_1}, \quad
    a_{(l+2)(l+2)} (x; \lambda) = (W(x) B^{-1})_{22}
    = \frac{W_{22} (x)}{b_2},
\end{equation}
and similarly 
\begin{equation} \label{second-terms}
    \sum_{k \notin \{(l+1), (l+2)\}} (a_{(l+2)k} (x; \lambda_2) - a_{(l+2)k} (x; \lambda)) h_{kj} (x; \lambda_2) 
    = - (\lambda_2 - \lambda) h_{2j}, \quad j = 1, 2, \dots, l.
\end{equation}
The terms (\ref{first-terms}) lead to an estimate by 
\begin{equation} \label{first-estimate}
    \Big(\frac{|W_{21} (x)|}{b_1} + \frac{|W_{22} (x)|}{b_2} \Big)
    |g_1 (x; \lambda)| \cdots |g_l (x; \lambda)|
    |h_1 (x; \lambda_2)| \cdots |h_l (x; \lambda_2)|,
\end{equation}
while the remaining terms (\ref{second-terms}) lead to the determinant 
\begin{equation} \label{full-determinant}
    \det
    \begin{pmatrix}
    g_{11} (\lambda) & \dots & g_{1l} (\lambda) 
    & h_{11} (\lambda_2) & \dots & h_{1l} (\lambda_2) \\
     \vdots & \vdots & \vdots & \vdots & \vdots & \vdots \\
    g_{l1} (\lambda) & \dots & g_{ll} (\lambda) 
    & h_{l1} (\lambda_2) & \dots & h_{ll} (\lambda_2) \\ 
    0 & \dots & 0
    & - (\lambda_2 - \lambda_1) h_{11} (\lambda_2)  & \dots & - (\lambda_2 - \lambda_1) h_{1l} (\lambda_2) \\
     0 & \dots & 0
    & - (\lambda_2 - \lambda) h_{21} (\lambda_2)  & \dots & - (\lambda_2 - \lambda) h_{2l} (\lambda_2) \\
     g_{(l+3)1} (\lambda) & \dots & g_{(l+3)l} (\lambda) 
    & h_{(l+3)1} (\lambda_2) & \dots & h_{(l+3)l} (\lambda_2) \\ 
     \vdots & \vdots & \vdots & \vdots & \vdots & \vdots \\
    g_{(2l)1} (\lambda) & \dots & g_{(2l)l} (\lambda) 
    & h_{(2l)1} (\lambda_2) & \dots & h_{(2l)l} (\lambda_2)
    \end{pmatrix}.
\end{equation}

To lowest order in $r$, we can compute this determinant with 
$\mathbf{G} (x; \lambda)$ and $\mathbf{H} (x; \lambda_2)$ 
respectively approximated by $\mathbf{G}_0 (x; \lambda)$ 
and $\mathbf{H}_0 (x; \lambda_2)$ as in (\ref{lowest-order-G})
and (\ref{lowest-order-H}). In this way, we obtain a determinant of 
the form 
\begin{equation*}
    \det
    \begin{pmatrix}
    I & 0 \\
    \tilde{\mathcal{G}} (x; \lambda) & \tilde{I}
    \end{pmatrix}
    = \det(\tilde{I}).
\end{equation*}
where the temporary notation $\tilde{\mathcal{G}} (x; \lambda)$
signifies the matrix obtained by taking the first two rows
of $\mathcal{G} (x; \lambda)$ to be identically zero while leaving 
all other rows unchanged, and 
likewise $\tilde{I}$ signifies the matrix obtained by taking the first two rows
of $I$ to be identically zero while leaving 
all other rows unchanged. Since $\det (\tilde{I}) = 0$, we can 
conclude that the full determinant (\ref{full-determinant}) is 
order $r$. 

These details have been carried out for the single term 
$\tilde{\omega}_2^{l+1} (x; \lambda, \lambda_1)$, and only
for the cases in which derivatives appear on one of the 
first $l+2$ rows. However, the analysis of the terms
with derivatives on the remaining rows, and the analysis 
of the remaining terms $\{\tilde{\omega}_2^{i} (x; \lambda, \lambda_1)\}_{i=l+2}^{2l}$
introduces no additional complications, and we can conclude that there 
exists a constant $K > 0$, independent of the values $\{b_i\}_{i=1}^l$,
so that 
\begin{equation*}
    \tilde{\omega}_2 ' (x; \lambda) \le K r,
\end{equation*}
and consequently 
\begin{equation*}
    |\frac{\tilde{\omega}_2 ' (x; \lambda)}{d (x; \lambda)}| \le \frac{K r}{d_{\min}}.
\end{equation*}

In our general invariance relation (\ref{invariance-criterion}), 
we can now take $C$ as specified in Lemma \ref{invariance-proposition1},
with 
\begin{equation*}
\begin{aligned}
    C_d &= \frac{l!C_A}{c_g^2} + \frac{l!}{c_h^2} \max_{x \in [0, 1]} \|A (x; \lambda_2)\| \\
    \delta &= \frac{K r}{d_{\min}}
\end{aligned}
\end{equation*}
keeping in mind that $C_a$ and $C_A$ can both be bounded independently
of the values $\{b_i\}_{i=1}^l$. Since $C$ can be taken 
independent of the values $\{b_i\}_{i=1}^l$, and $\delta$
can be taken as small as we like by decreasing 
$r$, we can ensure (\ref{invariance-criterion}) holds
so long as we can show that $\rho (0; \lambda)$ remains 
bounded away from 0 as $r$ decreases. 

For this final point, we recall that $\rho (0; \lambda)$
can be expressed as 
\begin{equation*}
    \rho (0; \lambda) = \frac{1}{2 d (0; \lambda)^2} (\tilde{\omega}_1 (0; \lambda)^2 + \tilde{\omega}_2 (0; \lambda)^2).
\end{equation*}
We can compute this value to lowest order in $r$ by using 
the frames $\mathbf{G} (0; \lambda) = {I \choose \Theta}$ and 
$\mathbf{H} (0; \lambda_2) = {0 \choose I}$. We see immediately
that 
\begin{equation*}
    \tilde{\omega}_1 (0; \lambda) 
    = \det
    \begin{pmatrix}
    I & 0 \\
    \Theta & I
    \end{pmatrix} = \det I = 1,
\end{equation*}
from which we can conclude that to lowest order 
in $r$
\begin{equation*}
    \rho (0; \lambda)
    \ge \frac{1}{2 d_{\min}^2}.
\end{equation*}
\end{proof}

\subsubsection{Example Case with Invariance}
\label{example2-section}

In this section, we will apply Theorem \ref{main-theorem} to 
(\ref{second-order-equation}) with $l = 2$, taking specifically 
$B$ to be the $2 \times 2$ identity matrix and
\begin{equation} \label{example2-coefficients}
    V (x) = \begin{pmatrix}
    10 \sin(10x) \cos(10x) & 25 \sin (10x) \\
    x (1-x) & 10 \cos (10x)
    \end{pmatrix},
    \quad
    W (x) = \begin{pmatrix}
    5x (1-x) & 0 \\
    0 & 5x (1-x)
    \end{pmatrix},
\end{equation}
along with Neumann boundary conditions at both $x = 0$ and $x = 1$.
For Neumann conditions, it's natural to take the frames for $\mathpzc{p}$ 
and $\mathpzc{q}$ to both be ${I \choose 0}$. For this example, we 
are not taking the entries of $B$ to be large, and in addition we are not 
using boundary conditions allowed by Proposition \ref{proposition2}, so 
we do not have an a priori guarantee of 
invariance. Nonetheless, we will (numerically) check 
invariance by computing $\rho (x; \lambda)$ throughout (a grid on)
the full Maslov box (including the interior). Indeed, one of our goals with this example is 
to illustrate that there is an enormous gap between systems for 
which we have rigorously verified invariance and systems for which 
invariance holds. 

We will count the number of eigenvalues the system 
(\ref{second-order-equation})--(\ref{example2-coefficients}) has
on the interval $[\lambda_1, \lambda_2] = [-5, 1]$. For this, 
we compute  $\ind (\mathpzc{g} (\cdot; -5), \mathpzc{h} (\cdot; 1); [0, 1])$,
and we find two crossing points, at about $x = .043$ and $x = .455$ 
(with a stepsize in the computation of $.001$). If the system is known to be invariant 
on $[0,1] \times [-5, 1]$ then we can conclude that 
(\ref{second-order-equation})--(\ref{example2-coefficients}) has
at least two eigenvalues on the interval $[-5, 1]$. This is the
most information that we can get out of Theorem \ref{main-theorem}
for this example, but computationally, we find approximately that 
\begin{equation*}
    \min_{\underset{\lambda \in [\lambda_1, \lambda_2]}{x \in [0, 1]}}
    \rho (x; \lambda) = .6279,
\end{equation*}
suggesting that invariance indeed holds. The full Maslov box 
for this example is depicted in Figure \ref{eg2_figure}. The 
eigenvalues are at roughly $\lambda = -1.385$ and $\lambda = .735$ 
(with a stepsize in the computation of $.0025$). 

\begin{figure}[ht] 
\begin{center}\includegraphics[%
  width=12cm,
  height=8cm]{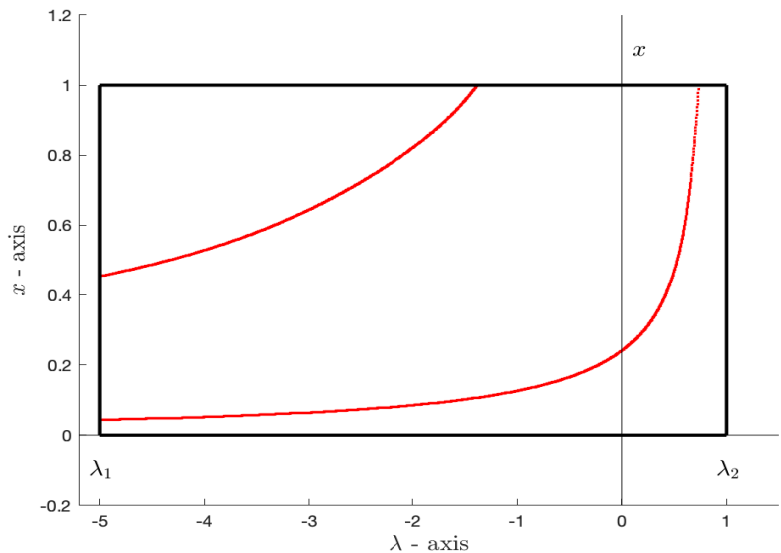}
\end{center}
\caption{Full Maslov Box and spectral curves for (\ref{second-order-equation})--(\ref{example2-coefficients}).  
\label{eg2_figure}}
\end{figure}

\subsubsection{Example Case without Invariance}
\label{example3-section}

An enormous amount remains to be said about invariance, and as a point
of interest, we compute the full Maslov box for a case in which invariance
fails to hold at precisely two points in the interior of the Maslov box. 
For this example, we'll take
(\ref{second-order-equation}) with $l = 2$, taking again 
$B$ to be the $2 \times 2$ identity matrix and choosing
\begin{equation} \label{example3-coefficients}
    V (x) = \begin{pmatrix}
    10 \sin(10x) \cos(10x) & 25 \sin (10x) \\
    x (1-x) & 10 \cos (10x)
    \end{pmatrix},
    \quad
    W (x) = \begin{pmatrix}
    5x (1-x) & 10 \sin(10x) \\
    10 \cos(10x) & 5x (1-x)
    \end{pmatrix},
\end{equation}
along with Neumann boundary conditions at both $x = 0$ and $x = 1$. 

For the system (\ref{second-order-equation})--(\ref{example3-coefficients}), 
we find by numerical computation that $\rho (x; \lambda)$ has two zeros 
in the interior of the Maslov box $[0, 1] \times [-5, 1]$, approximately
at the points $(.875, - 3.348)$ and $(.875, -.130)$. This suggest that 
invariance fails in this case. The full Maslov box for this example is 
depicted in Figure \ref{eg3-figure}. On the left-hand side of the figure, the
Maslov box is drawn for $[\lambda_1, \lambda_2] = [-5, 1]$, and we see
that the associated spectral curve is a loop contained entirely in the 
interior of the Maslov box, with left-most and right-most points 
corresponding precisely with zeros of $\rho (x; \lambda)$. Since no
spectral curves intersect the boundary of the Maslov box, it's clear
that $\mathfrak{m} = 0$, and it's interesting to understand how we 
can see this from the local considerations discussed in Section 
\ref{invariance-subsection}. To this end, we consider the contribution 
to $\mathfrak{m}$ from each of the points at which invariance is
lost. First, at $(.875, - 3.348)$, Figure  \ref{eg3-figure} suggests
that the spectral curve can be expressed as a functional relation 
$\lambda = \lambda (x)$, with $\lambda' (.875) = 0$, 
and we have precisely the situation 
of the middle plot in Figure \ref{invariance-figure} (with $\lambda$
now in place of $s$ and $x$ in place of $t$). As in the discussion 
in Section \ref{invariance-subsection}, we can conclude that the 
contribution to $\mathfrak{m}$ from this point is $+2$. The second
point at which invariance is lost is $(.875, -.130)$, and again we 
see that near this point the spectral curve can be expressed as a functional relation 
$\lambda = \lambda (x)$, with $\lambda' (.875) = 0$. In this case, 
we have precisely the situation 
of the left-side plot in Figure \ref{invariance-figure}, and can 
conclude that the 
contribution to $\mathfrak{m}$ from this point is $-2$. Since there 
are no other points of invariance, the total generalized Maslov index along the 
boundary is $\mathfrak{m} = +2 + (-2) = 0$. Using this information 
in our application of Theorem \ref{main-theorem}, we 
can write 
\begin{equation*}
    \mathcal{N}_{\#} ([-5, 1])
    \ge | \# \{x \in (0, L]: \mathpzc{g} (x; 0) \cap \mathpzc{q} \ne \{0\} \} + \mathfrak{m}|
    = 0.
\end{equation*}
In fact, it's clear from the full Maslov box that $\mathcal{N}_{\#} ([-5, 1]) = 0$.

Turning to the Maslov box on the right-hand side of Figure 
\ref{eg3-figure}, we see again that we have invariance along
the boundary of the Maslov box. (Here, we recall that invariance 
is only lost on the right and left endpoints of the spectral 
curves.) By monotonicity, each of the crossing points along the 
left shelf gives a contribution to the generalized Maslov index of $-1$, so
$\mathfrak{m} = -2$. Again, it's interesting to see that we can 
identify this value from local information. In this case, 
the only point in the Maslov box at which invariance is lost
is $(.875, -.130)$, and we have already seen that its contribution 
to $\mathfrak{m}$ will be $-2$. Since there are no other contributions
to $\mathfrak{m}$ in this case, we conclude that $\mathfrak{m} = -2$. 
Using this information 
in our application of Theorem \ref{main-theorem}, we 
can write 
\begin{equation*}
    \mathcal{N}_{\#} ([-3, 1])
    \ge | \# \{x \in (0, L]: \mathpzc{g} (x; 0) \cap \mathpzc{q} \ne \{0\} \} + (-2)|
    = 0.
\end{equation*}
In fact, it's clear from the full Maslov box that $\mathcal{N}_{\#} ([-3, 1]) = 0$.

\begin{figure}[ht] 
\begin{center}\includegraphics[%
  width=8.0cm,
  height=6.4cm]{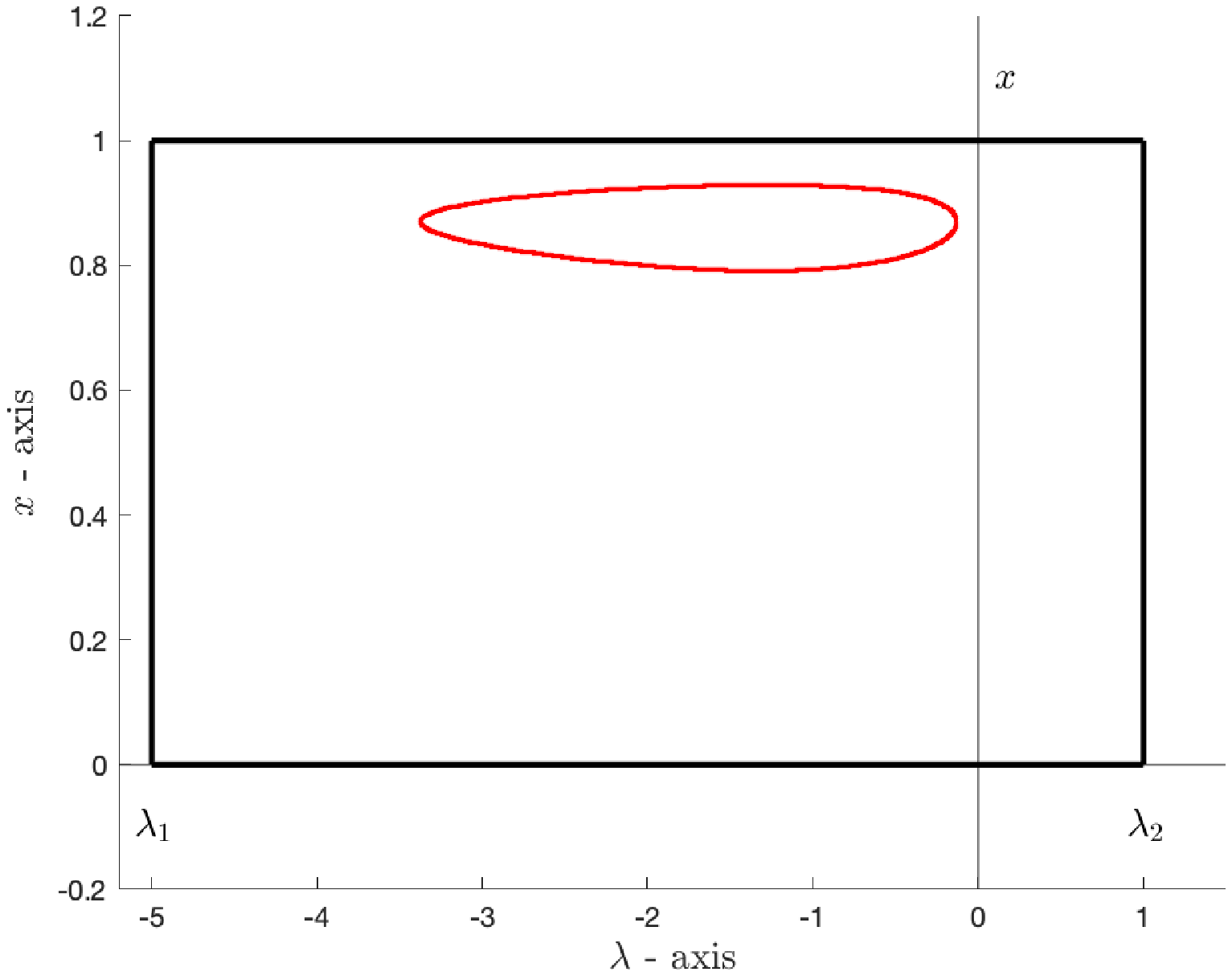}
\includegraphics[%
  width=8.0cm,
  height=6.4cm]{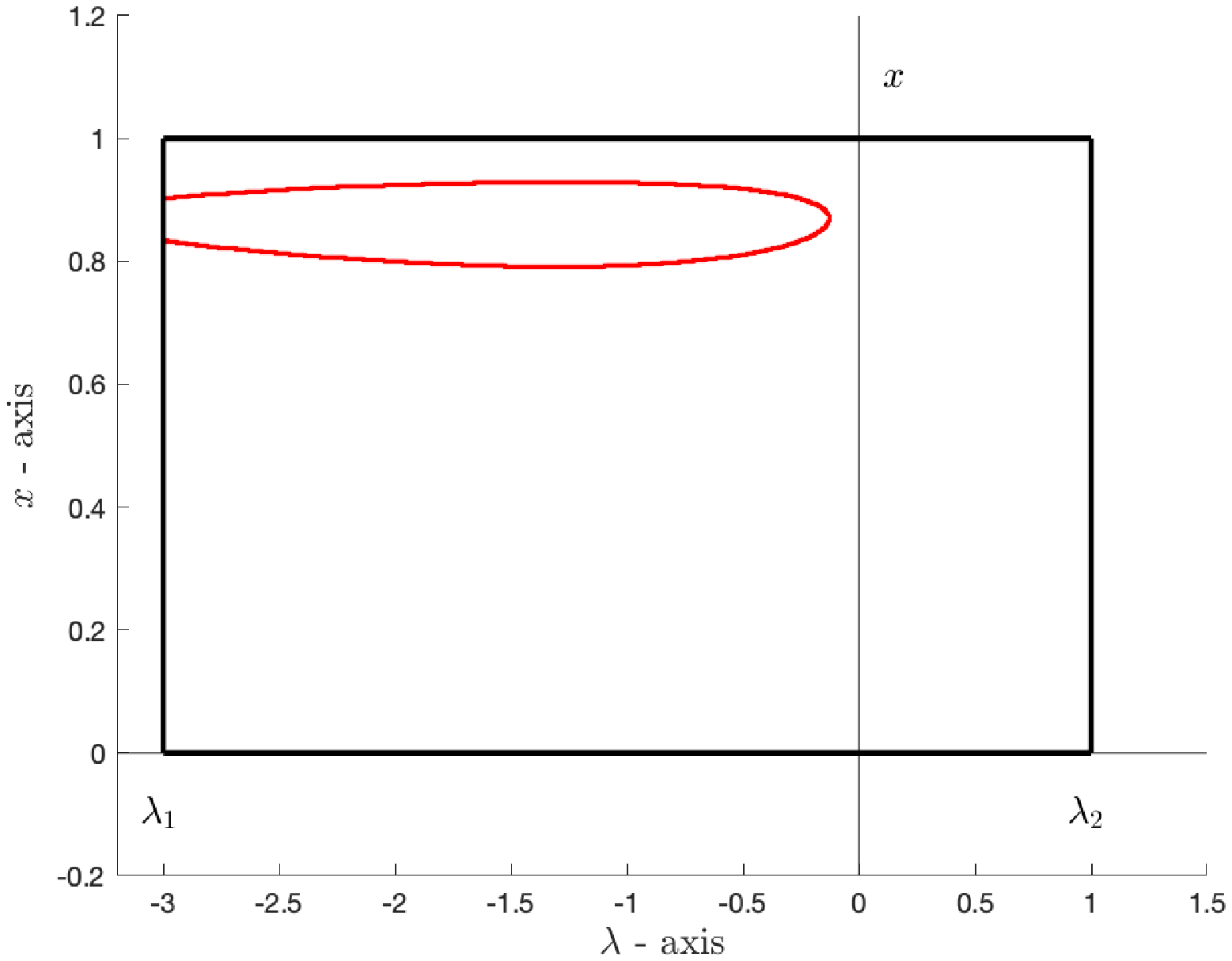}
\end{center}
\caption{Spectral curves for the system (\ref{second-order-equation})--(\ref{example3-coefficients}).
\label{eg3-figure}}
\end{figure}

\medskip
\noindent
{\it Acknowledgements}. The author is grateful to Graham Cox for patiently answering 
numerous questions about \cite{BCCJM2022}.

\end{document}